\documentclass[a4 ,10pt]{article}
\usepackage[english]{babel}
\usepackage[utf8]{inputenc}
\usepackage{amssymb}
\usepackage{amsmath}
\usepackage[all]{xy}
\usepackage{xcolor}
\usepackage{auto-pst-pdf}
\usepackage[usenames,dvipsnames]{pstricks}
\usepackage{epsfig}
\usepackage{pst-grad} 
\usepackage{pst-plot} 
\usepackage{graphicx}
\usepackage{graphics}
\usepackage{fancyhdr}
\newtheorem{theorem}{Theorem}[section]
\newtheorem{theorem and definition}{Theorem and definition}
\newtheorem{corollary}[theorem]{Corollary}
\newtheorem{definition}[theorem]{Definition}
\newtheorem{lemma}[theorem]{Lemma}
\newtheorem{proposition}[theorem]{Proposition}
\newenvironment{proof}[1][Proof]{\textbf{#1.} }{\ \rule{0.5em}{0.5em}}

\numberwithin{equation}{section}

\def\Z{\mathbb{ Z}}
\def\R{\mathbb{ R}}

\title{Some existence results for the modified binormal curvature flow equation}

\date{\today}

\providecommand{\classification }[2]{\textbf{\textit{AMS:}} #1}
\begin{document}

\author{
\renewcommand{\thefootnote}{\arabic{footnote}} Haidar Mohamad \footnotemark[1]}

\footnotetext[1]{Laboratoire Jacques-Louis Lions, Universit\'e Pierre et Marie
Curie, Bo\^ite Courrier 187,
75252 Paris Cedex 05, France. E-mail: {\tt tartousi@ann.jussieu.fr}}
\maketitle 

\begin{abstract}
We establish existence and uniqueness results for the modified binormal curvature flow equation that generalizes the binormal curvature flow
equation for a curve in $\R^3.$ In this generalization, the velocity of the curve is still directed along the binormal vector,
but the magnitude of the speed is allowed to depend on th parametrization of the curve, the time and the position of the point in the space. 
We achieve our objective via a generalized form of the Schr\"{o}dinger map equation. 
\end{abstract}
\vfill \classification{76B99 }

\newpage
\section{Introduction} 
\subsection{Statement of the main results}
The modified binormal curvature flow equation for $\gamma: [0,T[\times \mathbb{R}\rightarrow \mathbb{R}^3$ is
\begin{equation}\label{FCB}
\partial_t\gamma= g\left(\partial_x\gamma\wedge \partial^2_x\gamma\right),
\end{equation}
where $\wedge$ denotes the usual vector product in $\R^3$, $T\in \mathbb{R}_+^*\cup \{+\infty\},$ $x$ is  the arc-length parameter
of the curve  $\gamma(t,.)$ for all $t\in [0,T[$ and  $g$ is a real function depending on $t$, $x$ and eventually on $\gamma$.

The binormal curvature flow equation corresponds to $g\equiv 1$, namely
\begin{equation}\label{enBF}\tag{BF}
\partial_t \gamma = \partial_x \gamma \wedge \partial^2_x\gamma.
\end{equation}
Note that this  equation is compatible with the arc-length parametrization 
condition, since
$$
\partial_t \left| \partial_x\gamma\right|^2 = 2 \partial_x \gamma \cdot \partial_x\partial_t \gamma = 2 \partial_x 
\gamma \cdot \left( \partial_x \gamma \wedge \partial^3_x \gamma + \partial^2_x\gamma\wedge \partial^2_x\gamma \right) = 0.
$$
If $(\bf T,N,B)$ denotes the Frenet–Serret frame along the curve $\gamma,$ with $\bf T$ the tangent vector, $\bf N$ the normal vector
and $\bf B:= T \wedge N$ the binormal vector, then \eqref{enBF}  can be written in the geometrical form 
\begin{equation}\label{envit_corub}
 \partial_t\gamma = \kappa \bf B
\end{equation}
where $\kappa$ denotes the curvature of $\gamma$ at the considered point. Recall that the curvature  $\kappa$ and the  torsion $\tau$ are defined 
by the Frenet-Serret formula 
\begin{equation}\label{frenet}
 \partial_x\left(\begin{array}{lr} 
                 \bf T\\ \bf N\\\bf B
                 \end{array}\right)  =
\left(\begin{matrix} 
                  0 & \kappa & 0\\
                  -\kappa & 0 & \tau \\
                    0 & -\tau & 0 \\
                 \end{matrix} \right)
\left(\begin{array}{lr} 
                 \bf T\\\bf N\\\bf B
                 \end{array}\right).
\end{equation}
The formulation \eqref{envit_corub} indicates that the curve evolves with a velocity proportional to its curvature and oriented in the direction of the binormal vector. 
\begin{figure}[h]
 \centering
\scalebox{0.75} 
{
\begin{pspicture}(0,-2.79125)(5.741373,2.79125)
\psbezier[linewidth=0.02,arrowsize=0.05291667cm 2.0,arrowlength=1.4,arrowinset=0.4]{->}(3.3536108,-1.13125)(3.3536108,-1.93125)(1.1872215,-2.1160023)(0.59361076,-1.31125)(0.0,-0.5064977)(0.79456437,1.1308966)(1.6936108,1.56875)(2.592657,2.0066035)(3.7589324,1.4846574)(4.413611,0.72875)(5.0682893,-0.027157435)(5.721373,-1.6608737)(4.873611,-2.19125)
\psline[linewidth=0.02cm,arrowsize=0.05291667cm 2.0,arrowlength=1.4,arrowinset=0.4]{->}(2.4336107,1.74875)(4.1936107,1.70875)
\psline[linewidth=0.02cm,arrowsize=0.05291667cm 2.0,arrowlength=1.4,arrowinset=0.4]{->}(2.4736106,1.72875)(2.4536107,0.32875)
\psline[linewidth=0.02cm,arrowsize=0.05291667cm 2.0,arrowlength=1.4,arrowinset=0.4]{->}(2.4536107,1.74875)(3.6936107,2.34875)
\psline[linewidth=0.02cm,arrowsize=0.05291667cm 2.0,arrowlength=1.4,arrowinset=0.4]{->}(3.2736108,-1.43125)(2.3136108,-2.33125)
\psline[linewidth=0.02cm,arrowsize=0.05291667cm 2.0,arrowlength=1.4,arrowinset=0.4]{->}(3.2736108,-1.45125)(2.1336107,-1.09125)
\psline[linewidth=0.02cm,arrowsize=0.05291667cm 2.0,arrowlength=1.4,arrowinset=0.4]{->}(3.2336109,-1.45125)(3.4536107,-2.61125)
\usefont{T1}{ptm}{m}{n}
\rput(3.8,1.4){$\bf T$} 
\usefont{T1}{ptm}{m}{n}
\rput(2.6,0.59875){$\bf N$} 
\usefont{T1}{ptm}{m}{n}
\rput(3.6,2){$\bf B$} 
\usefont{T1}{ptm}{m}{n}
\rput(1.9045484,-2.3){$\bf T$} 
\usefont{T1}{ptm}{m}{n}
\rput(1.8729857,-1.4){$\bf N$}
\usefont{T1}{ptm}{m}{n}
\rput(3,-2.74125){$\bf B$} 
\rput(2.4,1.9){$\gamma$} 
\psline[linewidth=0.02cm,arrowsize=0.05291667cm 2.0,arrowlength=1.4,arrowinset=0.4]{->}(0.0,-1.80125)(2.4336107,1.74875)
\psline[linewidth=0.02cm,arrowsize=0.05291667cm 2.0,arrowlength=1.4,arrowinset=0.4]{->}(0.0,-1.80125)(0.02,2.57875)
\psline[linewidth=0.02cm,arrowsize=0.05291667cm 2.0,arrowlength=1.4,arrowinset=0.4]{->}(0.02,-1.78125)(1.14,-0.96125)
\psline[linewidth=0.02cm,arrowsize=0.05291667cm 2.0,arrowlength=1.4,arrowinset=0.4]{->}(0.0,-1.78125)(2.96,-3.00125)
\end{pspicture} 
}

\caption{ Frenet-Serret frame.}
\end{figure}
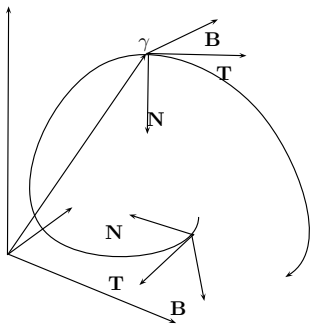
Taking into account the formulation \eqref{envit_corub}, Hasimoto \cite{Hasimoto} proved that if we  define the complex wave function $\Psi$ by
 \begin{equation}\label{eninconnu_Hasimoto}
 \Psi(t,x) = \kappa(t,x)\exp\left(i\int_0^x\tau(t,s)ds\right),
\end{equation}
where $\kappa$ and $\tau$ are the curvatures and the torsions for a family of curves animated by \eqref{enBF}, then $\Psi$ satisfies
the focusing cubic Schr\"odinger equation 
\begin{equation}\label{enHsimoto}
 i\partial_t \Psi+\partial^2_x \Psi +\frac{1}{2}\left(|\Psi(t,x)|^2+  A(t)\right)\Psi = 0,
\end{equation}
where $A(t)$ is a real function which can be removed by means of an integrating factor. The inverse transformation, when it
is well defined and in particular when $\Psi$ does not vanish,  allows to pass from the focusing cubic  Schr\"odinger equation to 
the binormal curvature flow equation.  

There is also a link between \eqref{enBF} and a class of particular solutions for the three-dimensional Gross-Pitaevskii equation (see \cite{BF_GP_3D})
$$i\partial_t \Psi+\Delta \Psi +(1-|\Psi|^2)\Psi = 0.$$ 
Such solutions behave like vortex filaments located along tubular neighborhoods of closed curves. It is expected in this case that the asymptotic limit 
of tubular evolution is given by the binormal curvature flow in the asymptotic limit where the section of the tubes tends to zero.
More specifically, if $\epsilon>0$ is a small parameter denoting the radius of a section of a tubular neighborhood, 
then it is expecting that the motion  of the curved axis of the tube is animated  by the equation (see for example  \cite{BF_GP_3D, LIA_S_J})
$$
\partial_t \gamma = G \kappa \bf B,
$$  
where 
$$G = \frac{\Gamma}{4\pi}\left(\log\left(\frac{1}{\epsilon}\right)+ O(1)\right),$$ 
and $\Gamma$ denotes the vortex intensity of the tube. 
\begin{figure}[h]
 \centering
 \scalebox{0.75} 
{
\begin{pspicture}(0,-2.61)(9.56,2.61)
\pscustom[linewidth=0.02]
{
\newpath
\moveto(0.28,-0.25)
\lineto(0.84,0.16)
\curveto(1.12,0.365)(1.69,0.675)(1.98,0.78)
\curveto(2.27,0.885)(2.925,1.005)(3.29,1.02)
\curveto(3.655,1.035)(4.415,1.09)(4.81,1.13)
\curveto(5.205,1.17)(5.885,1.31)(6.17,1.41)
\curveto(6.455,1.51)(6.88,1.765)(7.02,1.92)
\curveto(7.16,2.075)(7.315,2.245)(7.36,2.29)
}
\psbezier[linewidth=0.02](0.2,-0.43)(0.2,-1.23)(2.18,-2.59)(2.18,-1.79)(2.18,-0.99)(0.2,0.37)(0.2,-0.43)
\pscustom[linewidth=0.02]
{
\newpath
\moveto(2.12,-2.03)
\lineto(2.16,-2.01)
\curveto(2.18,-2.0)(2.22,-1.975)(2.24,-1.96)
\curveto(2.26,-1.945)(2.305,-1.91)(2.33,-1.89)
\curveto(2.355,-1.87)(2.405,-1.835)(2.43,-1.82)
\curveto(2.455,-1.805)(2.535,-1.755)(2.59,-1.72)
\curveto(2.645,-1.685)(2.77,-1.605)(2.84,-1.56)
\curveto(2.91,-1.515)(3.06,-1.435)(3.14,-1.4)
\curveto(3.22,-1.365)(3.375,-1.3)(3.45,-1.27)
\curveto(3.525,-1.24)(3.7,-1.185)(3.8,-1.16)
\curveto(3.9,-1.135)(4.085,-1.09)(4.17,-1.07)
\curveto(4.255,-1.05)(4.45,-1.015)(4.56,-1.0)
\curveto(4.67,-0.985)(4.895,-0.96)(5.01,-0.95)
\curveto(5.125,-0.94)(5.35,-0.915)(5.46,-0.9)
\curveto(5.57,-0.885)(5.805,-0.845)(5.93,-0.82)
\curveto(6.055,-0.795)(6.325,-0.725)(6.47,-0.68)
\curveto(6.615,-0.635)(6.925,-0.525)(7.09,-0.46)
\curveto(7.255,-0.395)(7.595,-0.215)(7.77,-0.1)
\curveto(7.945,0.015)(8.235,0.255)(8.35,0.38)
\curveto(8.465,0.505)(8.6,0.7)(8.62,0.77)
\curveto(8.64,0.84)(8.675,0.93)(8.69,0.95)
}
\pscustom[linewidth=0.02]
{
\newpath
\moveto(8.68,0.87)
\lineto(8.69,0.91)
\curveto(8.695,0.93)(8.705,0.97)(8.71,0.99)
\curveto(8.715,1.01)(8.72,1.055)(8.72,1.08)
\curveto(8.72,1.105)(8.72,1.155)(8.72,1.18)
\curveto(8.72,1.205)(8.725,1.255)(8.73,1.28)
\curveto(8.735,1.305)(8.745,1.355)(8.75,1.38)
\curveto(8.755,1.405)(8.755,1.45)(8.75,1.47)
\curveto(8.745,1.49)(8.735,1.53)(8.73,1.55)
\curveto(8.725,1.57)(8.72,1.615)(8.72,1.64)
\curveto(8.72,1.665)(8.71,1.71)(8.7,1.73)
\curveto(8.69,1.75)(8.67,1.795)(8.66,1.82)
\curveto(8.65,1.845)(8.635,1.89)(8.63,1.91)
\curveto(8.625,1.93)(8.605,1.965)(8.59,1.98)
\curveto(8.575,1.995)(8.555,2.03)(8.55,2.05)
\curveto(8.545,2.07)(8.53,2.105)(8.52,2.12)
\curveto(8.51,2.135)(8.49,2.17)(8.48,2.19)
\curveto(8.47,2.21)(8.45,2.245)(8.44,2.26)
\curveto(8.43,2.275)(8.4,2.295)(8.38,2.3)
\curveto(8.36,2.305)(8.33,2.325)(8.32,2.34)
\curveto(8.31,2.355)(8.28,2.38)(8.26,2.39)
\curveto(8.24,2.4)(8.2,2.415)(8.18,2.42)
\curveto(8.16,2.425)(8.12,2.435)(8.1,2.44)
\curveto(8.08,2.445)(8.045,2.46)(8.03,2.47)
\curveto(8.015,2.48)(7.98,2.495)(7.96,2.5)
\curveto(7.94,2.505)(7.895,2.515)(7.87,2.52)
\curveto(7.845,2.525)(7.8,2.525)(7.78,2.52)
\curveto(7.76,2.515)(7.715,2.51)(7.69,2.51)
\curveto(7.665,2.51)(7.63,2.495)(7.62,2.48)
\curveto(7.61,2.465)(7.585,2.44)(7.57,2.43)
\curveto(7.555,2.42)(7.52,2.405)(7.5,2.4)
\curveto(7.48,2.395)(7.455,2.38)(7.45,2.37)
\curveto(7.445,2.36)(7.43,2.335)(7.42,2.32)
\curveto(7.41,2.305)(7.395,2.29)(7.38,2.29)
}
\pscustom[linewidth=0.02,linestyle=dashed,dash=0.16cm 0.16cm]
{
\newpath
\moveto(8.66,0.85)
\lineto(8.62,0.83)
\curveto(8.6,0.82)(8.56,0.8)(8.54,0.79)
\curveto(8.52,0.78)(8.49,0.755)(8.48,0.74)
\curveto(8.47,0.725)(8.44,0.705)(8.42,0.7)
\curveto(8.4,0.695)(8.355,0.69)(8.33,0.69)
\curveto(8.305,0.69)(8.26,0.685)(8.24,0.68)
\curveto(8.22,0.675)(8.175,0.67)(8.15,0.67)
\curveto(8.125,0.67)(8.08,0.675)(8.06,0.68)
\curveto(8.04,0.685)(8.0,0.7)(7.98,0.71)
\curveto(7.96,0.72)(7.915,0.74)(7.89,0.75)
\curveto(7.865,0.76)(7.82,0.775)(7.8,0.78)
\curveto(7.78,0.785)(7.745,0.8)(7.73,0.81)
\curveto(7.715,0.82)(7.685,0.84)(7.67,0.85)
\curveto(7.655,0.86)(7.625,0.885)(7.61,0.9)
\curveto(7.595,0.915)(7.57,0.945)(7.56,0.96)
\curveto(7.55,0.975)(7.525,1.01)(7.51,1.03)
\curveto(7.495,1.05)(7.475,1.09)(7.47,1.11)
\curveto(7.465,1.13)(7.455,1.17)(7.45,1.19)
\curveto(7.445,1.21)(7.43,1.245)(7.42,1.26)
\curveto(7.41,1.275)(7.39,1.31)(7.38,1.33)
\curveto(7.37,1.35)(7.355,1.39)(7.35,1.41)
\curveto(7.345,1.43)(7.34,1.475)(7.34,1.5)
\curveto(7.34,1.525)(7.34,1.575)(7.34,1.6)
\curveto(7.34,1.625)(7.345,1.67)(7.35,1.69)
\curveto(7.355,1.71)(7.36,1.755)(7.36,1.78)
\curveto(7.36,1.805)(7.36,1.855)(7.36,1.88)
\curveto(7.36,1.905)(7.365,1.95)(7.37,1.97)
\curveto(7.375,1.99)(7.385,2.03)(7.39,2.05)
\curveto(7.395,2.07)(7.405,2.115)(7.41,2.14)
\curveto(7.415,2.165)(7.425,2.215)(7.43,2.24)
\curveto(7.435,2.265)(7.45,2.305)(7.46,2.32)
\curveto(7.47,2.335)(7.49,2.365)(7.5,2.38)
}
\pscustom[linewidth=0.02]
{
\newpath
\moveto(4.24,-1.07)
\lineto(4.26,-1.04)
\curveto(4.27,-1.025)(4.295,-1.0)(4.31,-0.99)
\curveto(4.325,-0.98)(4.365,-0.97)(4.39,-0.97)
\curveto(4.415,-0.97)(4.45,-0.955)(4.46,-0.94)
\curveto(4.47,-0.925)(4.5,-0.895)(4.52,-0.88)
\curveto(4.54,-0.865)(4.57,-0.835)(4.58,-0.82)
\curveto(4.59,-0.805)(4.61,-0.77)(4.62,-0.75)
\curveto(4.63,-0.73)(4.645,-0.69)(4.65,-0.67)
\curveto(4.655,-0.65)(4.665,-0.605)(4.67,-0.58)
\curveto(4.675,-0.555)(4.68,-0.505)(4.68,-0.48)
\curveto(4.68,-0.455)(4.68,-0.405)(4.68,-0.38)
\curveto(4.68,-0.355)(4.675,-0.31)(4.67,-0.29)
\curveto(4.665,-0.27)(4.655,-0.225)(4.65,-0.2)
\curveto(4.645,-0.175)(4.635,-0.13)(4.63,-0.11)
\curveto(4.625,-0.09)(4.615,-0.05)(4.61,-0.03)
\curveto(4.605,-0.01)(4.59,0.025)(4.58,0.04)
\curveto(4.57,0.055)(4.55,0.085)(4.54,0.1)
\curveto(4.53,0.115)(4.515,0.15)(4.51,0.17)
\curveto(4.505,0.19)(4.485,0.225)(4.47,0.24)
\curveto(4.455,0.255)(4.425,0.29)(4.41,0.31)
\curveto(4.395,0.33)(4.365,0.365)(4.35,0.38)
\curveto(4.335,0.395)(4.31,0.425)(4.3,0.44)
\curveto(4.29,0.455)(4.26,0.49)(4.24,0.51)
\curveto(4.22,0.53)(4.19,0.575)(4.18,0.6)
\curveto(4.17,0.625)(4.14,0.67)(4.12,0.69)
\curveto(4.1,0.71)(4.065,0.74)(4.05,0.75)
\curveto(4.035,0.76)(4.005,0.78)(3.99,0.79)
\curveto(3.975,0.8)(3.94,0.825)(3.92,0.84)
\curveto(3.9,0.855)(3.86,0.875)(3.84,0.88)
\curveto(3.82,0.885)(3.78,0.895)(3.76,0.9)
\curveto(3.74,0.905)(3.695,0.925)(3.67,0.94)
\curveto(3.645,0.955)(3.6,0.975)(3.58,0.98)
\curveto(3.56,0.985)(3.52,0.995)(3.5,1.0)
\curveto(3.48,1.005)(3.44,1.015)(3.42,1.02)
}
\pscustom[linewidth=0.02,linestyle=dashed,dash=0.16cm 0.16cm]
{
\newpath
\moveto(3.42,1.01)
\lineto(3.37,1.01)
\curveto(3.345,1.01)(3.305,0.995)(3.29,0.98)
\curveto(3.275,0.965)(3.245,0.94)(3.23,0.93)
\curveto(3.215,0.92)(3.175,0.9)(3.15,0.89)
\curveto(3.125,0.88)(3.085,0.855)(3.07,0.84)
\curveto(3.055,0.825)(3.025,0.795)(3.01,0.78)
\curveto(2.995,0.765)(2.965,0.735)(2.95,0.72)
\curveto(2.935,0.705)(2.91,0.665)(2.9,0.64)
\curveto(2.89,0.615)(2.87,0.57)(2.86,0.55)
\curveto(2.85,0.53)(2.83,0.485)(2.82,0.46)
\curveto(2.81,0.435)(2.805,0.39)(2.81,0.37)
\curveto(2.815,0.35)(2.825,0.31)(2.83,0.29)
\curveto(2.835,0.27)(2.845,0.23)(2.85,0.21)
\curveto(2.855,0.19)(2.87,0.145)(2.88,0.12)
\curveto(2.89,0.095)(2.905,0.045)(2.91,0.02)
\curveto(2.915,-0.005)(2.925,-0.055)(2.93,-0.08)
\curveto(2.935,-0.105)(2.955,-0.15)(2.97,-0.17)
\curveto(2.985,-0.19)(3.015,-0.23)(3.03,-0.25)
\curveto(3.045,-0.27)(3.065,-0.315)(3.07,-0.34)
\curveto(3.075,-0.365)(3.095,-0.405)(3.11,-0.42)
\curveto(3.125,-0.435)(3.15,-0.465)(3.16,-0.48)
\curveto(3.17,-0.495)(3.195,-0.525)(3.21,-0.54)
\curveto(3.225,-0.555)(3.26,-0.585)(3.28,-0.6)
\curveto(3.3,-0.615)(3.34,-0.645)(3.36,-0.66)
\curveto(3.38,-0.675)(3.43,-0.695)(3.46,-0.7)
\curveto(3.49,-0.705)(3.535,-0.72)(3.55,-0.73)
\curveto(3.565,-0.74)(3.595,-0.76)(3.61,-0.77)
\curveto(3.625,-0.78)(3.655,-0.8)(3.67,-0.81)
\curveto(3.685,-0.82)(3.72,-0.835)(3.74,-0.84)
\curveto(3.76,-0.845)(3.795,-0.86)(3.81,-0.87)
\curveto(3.825,-0.88)(3.86,-0.895)(3.88,-0.9)
\curveto(3.9,-0.905)(3.93,-0.925)(3.94,-0.94)
\curveto(3.95,-0.955)(3.975,-0.98)(3.99,-0.99)
\curveto(4.005,-1.0)(4.045,-1.01)(4.07,-1.01)
\curveto(4.095,-1.01)(4.145,-1.01)(4.17,-1.01)
\curveto(4.195,-1.01)(4.24,-1.005)(4.26,-1.0)
\curveto(4.28,-0.995)(4.31,-0.99)(4.34,-0.99)
}
\pscustom[linewidth=0.02,linestyle=dashed,dash=0.16cm 0.16cm]
{
\newpath
\moveto(1.22,-1.11)
\lineto(1.23,-1.09)
\curveto(1.235,-1.08)(1.265,-1.05)(1.29,-1.03)
\curveto(1.315,-1.01)(1.38,-0.965)(1.42,-0.94)
\curveto(1.46,-0.915)(1.53,-0.865)(1.56,-0.84)
\curveto(1.59,-0.815)(1.665,-0.76)(1.71,-0.73)
\curveto(1.755,-0.7)(1.835,-0.645)(1.87,-0.62)
\curveto(1.905,-0.595)(1.975,-0.55)(2.01,-0.53)
\curveto(2.045,-0.51)(2.12,-0.47)(2.16,-0.45)
\curveto(2.2,-0.43)(2.28,-0.385)(2.32,-0.36)
\curveto(2.36,-0.335)(2.44,-0.295)(2.48,-0.28)
\curveto(2.52,-0.265)(2.605,-0.235)(2.65,-0.22)
\curveto(2.695,-0.205)(2.78,-0.18)(2.82,-0.17)
\curveto(2.86,-0.16)(2.955,-0.135)(3.01,-0.12)
\curveto(3.065,-0.105)(3.17,-0.08)(3.22,-0.07)
\curveto(3.27,-0.06)(3.365,-0.04)(3.41,-0.03)
\curveto(3.455,-0.02)(3.535,-0.005)(3.57,0.0)
\curveto(3.605,0.005)(3.685,0.01)(3.73,0.01)
\curveto(3.775,0.01)(3.865,0.015)(3.91,0.02)
\curveto(3.955,0.025)(4.035,0.03)(4.07,0.03)
\curveto(4.105,0.03)(4.18,0.03)(4.22,0.03)
\curveto(4.26,0.03)(4.325,0.03)(4.35,0.03)
\curveto(4.375,0.03)(4.43,0.03)(4.46,0.03)
\curveto(4.49,0.03)(4.565,0.03)(4.61,0.03)
\curveto(4.655,0.03)(4.74,0.035)(4.78,0.04)
\curveto(4.82,0.045)(4.905,0.06)(4.95,0.07)
\curveto(4.995,0.08)(5.08,0.095)(5.12,0.1)
\curveto(5.16,0.105)(5.255,0.12)(5.31,0.13)
\curveto(5.365,0.14)(5.445,0.155)(5.47,0.16)
\curveto(5.495,0.165)(5.565,0.185)(5.61,0.2)
\curveto(5.655,0.215)(5.74,0.245)(5.78,0.26)
\curveto(5.82,0.275)(5.9,0.305)(5.94,0.32)
\curveto(5.98,0.335)(6.06,0.365)(6.1,0.38)
\curveto(6.14,0.395)(6.21,0.42)(6.24,0.43)
\curveto(6.27,0.44)(6.355,0.48)(6.41,0.51)
\curveto(6.465,0.54)(6.56,0.595)(6.6,0.62)
\curveto(6.64,0.645)(6.73,0.695)(6.78,0.72)
\curveto(6.83,0.745)(6.945,0.8)(7.01,0.83)
\curveto(7.075,0.86)(7.195,0.92)(7.25,0.95)
\curveto(7.305,0.98)(7.405,1.045)(7.45,1.08)
\curveto(7.495,1.115)(7.575,1.185)(7.61,1.22)
\curveto(7.645,1.255)(7.725,1.325)(7.77,1.36)
\curveto(7.815,1.395)(7.93,1.485)(8.0,1.54)
\curveto(8.07,1.595)(8.17,1.685)(8.2,1.72)
\curveto(8.23,1.755)(8.305,1.825)(8.35,1.86)
\curveto(8.395,1.895)(8.485,1.955)(8.53,1.98)
\curveto(8.575,2.005)(8.625,2.035)(8.64,2.05)
}
\psline[linewidth=0.02cm,arrowsize=0.05291667cm 2.0,arrowlength=1.4,arrowinset=0.4]{->}(8.58,2.03)(9.54,2.59)
\psline[linewidth=0.02cm,arrowsize=0.05291667cm 2.0,arrowlength=1.4,arrowinset=0.4]{->}(0.0,-2.17)(1.24,-1.09)
\psline[linewidth=0.02cm,linestyle=dashed,dash=0.16cm 0.16cm](3.64,-0.01)(3.28,0.97)
\psline[linewidth=0.02cm](1.22,-1.13)(0.3,-0.25)
\psline[linewidth=0.02cm,linestyle=dashed,dash=0.16cm 0.16cm](8.04,1.57)(7.5,2.37)
\usefont{T1}{ptm}{m}{n}
\rput(3.3,0.1){$\epsilon$} 
\usefont{T1}{ptm}{m}{n}
\rput(0.8,-1.1){$\epsilon$} 
\end{pspicture} 
}
\caption{Approximation (LIA).}
\end{figure}
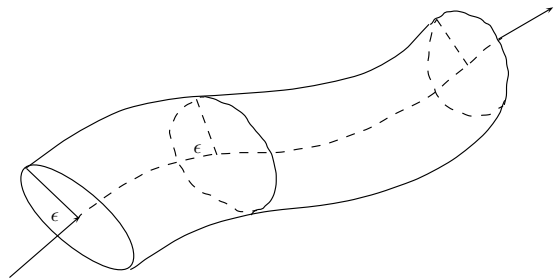

This asymptotic limit is  sometimes used for the study of the three-dimensional Euler equation, where it is called the local induction approximation (LIA). In this context, it
was formally derived by Da Rios \cite{Rios} in 1906 and rediscovered by Arms and Hama \cite{LIA_1} in 1965. Further analysis concerning the limitation of this model
were realized in \cite{LIA_2, LIA_3}.    

Using equation (\ref{enHsimoto}) (which is integrable), Hasimoto wrote an explicit formula for a solution of \eqref{enBF} that
 has constant torsion $\tau$ and a curvature $\kappa$ given by $\kappa(t,x) =2\nu \mathop{\rm sech}(\nu(x-2\tau t))$, where $\nu$ is a constant.
Another family of exact solutions for \eqref{enBF} were studied in \cite{GRV, Lak_Dan}.  Their curvature and torsion are given by
$$\kappa(t,x)= \frac{a}{\sqrt{t}}, \tau(t,x)= \frac{x}{2t},\quad t>0,$$ 
with $a\in \R$. In terms of $\gamma$, there exists a smooth function $G_a$ such that 
$$\gamma(t,x) = \sqrt{t}G_a(x/\sqrt{t}).$$ This solution satisfies formally at $t=0$ 
$$\gamma(0,x)= xA^-1_{]-\infty, 0[}+xA^+1_{]0,-\infty[},$$
where $A^\pm$ are two unit vectors in $\R^3$ and $\sin((\widehat{A^+, A^-})/2) = \exp(-\pi a^2/2)$. Hence $\gamma(0,.)$ is an infinite broken line. 
Rigorous studies of the behavior of these solutions and 
some of their perturbations were realized in \cite{Bca_Vga2, Bca_Vga1, Bca_Vga4} and numerical simulations in \cite{ Numerical_LIA2, Numerical_LIA1}.
Existence and stability properties for weak solutions of \eqref{enBF} corresponding to closed curves with corners were studied in \cite{LIA_5, LIA_S_J}.

Our goal in this paper is to prove existence results for the Cauchy problem associated to a generalization of equation \eqref{enBF}.  
More specifically, while in the formulation \eqref{enBF} the velocity is proportional to the  curvature with an identical coefficient at every point of the curve,
We generalize (\ref{enBF}) by assuming that the proportionality coefficient depends on $t$, on the arc-length parameter $x$, and possibly on
the position of the point in the space $\gamma(t,x)$. In other words,
 \begin{equation}\label{envit_corub_modif}
 \partial_t\gamma = g\kappa B,
\end{equation} 
with $g=g(t,x,\gamma(t,x))$, or equivalently, 
\begin{equation}\label{envit_corub_2}
 \partial_t\gamma = g\partial_x \gamma \wedge \partial^2_x\gamma.
\end{equation}
Note that this equation preserves the arc-length, as does the equation \eqref{enBF}, since   
$$
\partial_t \left| \partial_x\gamma\right|^2 = 2 \partial_x \gamma \cdot \partial_x\partial_t \gamma = 2 \partial_x
\gamma \cdot \bigl( g \partial_x \gamma \wedge \partial^3_x \gamma + g \partial^2_x\gamma\wedge \partial^2_x\gamma + 
(\partial_x g + \nabla_\gamma g \cdot \partial_x \gamma) \partial_x \gamma \wedge \partial^2_x \gamma \bigr) = 0.
$$
To simplify the statements of our results, we assume that $g$ is smooth, that all of its derivatives are bounded, and that there exists some positive constant $\alpha$ such that $g\geq \alpha.$  

When $g= g(t,x)$ does not depend on $\gamma$, by letting $u = \partial_x\gamma$
and deriving (\ref{envit_corub_2}) with respect to $x$ we get, at least formally,  
\begin{equation}\label{LIA}
\partial_t u=\partial_x\left(u\wedge g\partial_x u\right)=u\wedge\Delta_g(u),
\end{equation}
where $\Delta_g(u) \equiv \partial_x\left(g(x)\partial_x u\right)$. Note that if $g$ depends explicitly on $\gamma$, equation (\ref{LIA}) is no longer formed by $u$. 
We will use this last equation together with Lemma \ref{FCB_LIA} below, to solve the Cauchy problem associated to \eqref{envit_corub_2}.
\begin{lemma}\label{FCB_LIA}
Let $I$ be an interval containing zero. Assume that $g\equiv g(t,x).$ Let $u \in L^\infty(I,H^1_{loc}(\mathbb{R},S^2))$  
be a weak solution of (\ref{LIA}) such that $u \in L^\infty(I, L^2(\R))$ , and let $\Gamma_u$ be
the function defined by
\begin{equation}\label{enGamma}
\Gamma_u(t,x)=\int_0^x u(t,z)dz \quad \text{for all}\quad (t,x)\in I\times\mathbb{R} .
\end{equation}
Then there exists a unique continuous function $c_u: I\rightarrow \mathbb{R}^3$ satisfying $c_u(0)= 0$  
such that the function $\gamma_u\in L^\infty(I,H^2_{loc}(\mathbb{R},\mathbb{R}^3)),$ defined by
$$\gamma_u(t,x)= \Gamma_u(t,x)+c_u(t),$$
is a solution of  (\ref{envit_corub_2}) on $I\times \mathbb{R}.$
\end{lemma}
A similar result was proved in \cite{LIA_S_J}, when $u(t,.)$ is $l-$periodic.  

The case $g \equiv 1$ in \eqref{LIA} corresponds to the well-known equation for Schr\"{o}dinger maps with values in $S^2$ 
($ | u | = 1 $ is due to the arc-length parameterization) 
\begin{equation}\label{envit_corub_4}
\partial_tu =  u \wedge \partial^2_xu.
\end{equation}
The corresponding Cauchy problem was studied by Guo and Zhou \cite{Zhou_Guo} in 1984  when $u(t,.)$ is defined on some segment of  $\mathbb{R}$,  
and by Bardos, Sulem and Sulem \cite{LIA_4} in 1986 when $u(t,.)$ is defined on $\mathbb{R}^N$ ($N \geq 1$). 
They have proved that if the initial data has its gradient in $L^2$, then \eqref{envit_corub_4} has a global weak solution in $L^\infty(H^1_{loc}).$ 
Our approach is similar to theirs an uses a semi-discrete scheme  (discrete in space and continuous in time), some particular bounds on the discrete problems 
and a compactness argument. In this way, we prove the
\begin{theorem}\label{th_ex_np}
Let $u_0: \mathbb{R}\rightarrow S^2$ be such that $\frac{du_0}{dx}\in L^2(\mathbb{R}),$ $T>0$ and let 
$g\in W^{1,\infty}(\R^+,L^\infty(\mathbb{R}))$ be such that  $g\geq\alpha$ for some constant $\alpha>0$. Then equation (\ref{LIA}) has
a solution  $u \in L^\infty(0,T,H_{loc}^1(\mathbb{R},S^2))$ with $u(0,.)=u_0$ and $\partial_xu \in L^\infty(0,T,L^2(\mathbb{R}))$. 
Moreover, if $g=g(x),$ then $\partial_xu \in L^\infty(\R^+,L^2(\mathbb{R}))$.
\end{theorem}

Lemma \ref{FCB_LIA} allows to translate this result into a result for \eqref{envit_corub_2}. A similar result holds for periodic solutions in $x$.
Specifically,
\begin{theorem}\label{th_ex_p}
Let $l>0$, $T>0$ and $\mathbb{T}^l:=\mathbb{R}/l\mathbb{Z}.$ Let $u_0\in H^1(\mathbb{T}^l, S^2)$ and let 
$g\in W^{1,\infty}(\R^+,L^\infty(\mathbb{T}^l))$ be given such that there exists $\alpha>0$ with $g\geq\alpha.$ 
Then the equation (\ref{LIA}) has a solution  $u \in L^\infty(0,T,H^1(\mathbb{T}^l,S^2))$ with $u(0,.)=u_0.$
Moreover, if $g=g(x),$ then $u\in L^\infty(\R^+,H^1(\mathbb{T}^l,S^2)).$
\end{theorem}

Concerning the strong solutions, we will prove the following result. 
\begin{theorem}[Local well-posedness for regular solutions]\label{th_ex_npf}
Let  $u_0: \mathbb{R}\rightarrow S^2$ be such that $\frac{du_0}{dx} \in H^2(\mathbb{R}),$ and let $g\in W^{1,\infty}(\R^+,W^{3,\infty}(\mathbb{R})).$  
Assume that there exists $\alpha>0$ with $g\geq\alpha.$ Then there exists $T_1=T_1(g,u_0)>0$ such that equation (\ref{LIA}) has a unique solution
$u \in L^\infty(0,T_1,H_{loc}^3(\mathbb{R}))$ with $u(0,.)=u_0$ and $\partial_xu \in L^\infty(0,T_1,H^2(\mathbb{R}))$.
\end{theorem}

The uniqueness is obtained from the following  quantitative comparison theorem:
\begin{theorem}\label{th_un}
Let $T>0$ and let $g:\mathbb{R}\rightarrow\mathbb{R}$ be a function  verifying the assumptions of Theorem \ref{th_ex_npf}. Let
$u$ and $\tilde{u}$ be two  solutions for (\ref{LIA}) with  initial datum  $u_0, \tilde{u}_0: \mathbb{R}\rightarrow S^2$, respectively.
Assume that that $u_0-\tilde{u}_0 \in H^2(\R)$ and that $\partial_x u, \partial_x\tilde{u} \in L^\infty(0,T, H^2(\mathbb{R})).$  
Then there exists two positive constants $C_1,C_2$, depending on $g,$ $T$, and the $H^2$ norm  of 
$\frac{\partial u_0}{\partial x}$ and $\frac{\partial \tilde{u}_0}{\partial x}$, such that
$$\|u(t,.)-\tilde{u}(t,.)\|_{H^1(\mathbb{R})}\leq C_1\|u_0-\tilde{u}_0\|_{H^1(\mathbb{R})},$$
$$\|u(t,.)-\tilde{u}(t,.)\|_{H^2(\mathbb{R})}\leq C_2\|u_0-\tilde{u}_0\|_{H^2(\mathbb{R})},$$
for almost every $t\in ]0,T[.$
\end{theorem}

Concerning the case $g=g(t,x,\gamma)$, we have 
\begin{theorem}\label{ex_fcb_g(gamma)}
Assume that $g=g(t,x,\gamma)$ for some function $g \in W^{1,\infty}(\R^+, W^{2,\infty}(\mathbb{R}\times\mathbb{R}^3))$ and that there  exists $\alpha>0$ such that
$g\geq\alpha.$  Let $\gamma_0: \mathbb{R}\rightarrow\mathbb{R}^3$ be such that $\frac{d^2\gamma_0}{dx^2}\in H^1(\mathbb{R}).$
Then there exists $T_1 =T_1(g,\gamma_0)$ such that equation (\ref{FCB}) has a solution $\gamma \in L^\infty(0,T_1,H^3_{loc}(\mathbb{R}))$
with $\gamma(0,.)= \gamma_0$ and $\partial^2_x \gamma\in L^\infty(0,T_1,H^1(\mathbb{R}))$.
\end{theorem}

\subsection{Reconstruction of the flow $\gamma$}  
Let $I\subset \mathbb{R}^+$ be an interval containing $0,$ let $u\in L^\infty(I,H^1_{loc}(\mathbb{R}))$ be a solution for (\ref{LIA}) and let
the function $\Gamma_u\in L^\infty(I,H^2_{loc}(\mathbb{R}))$ be defined by
\begin{equation}\label{Gamma_u}
\Gamma_u(t,x)=\int_0^xu(t,z)dz.
\end{equation}
In the sense of distributions on $I\times\mathbb{R}$, We have
\begin{equation}\label{Gamma_u_1}
\partial_x\left(\partial_t\Gamma_u-g\partial_x\Gamma_u\wedge \partial^2_x\Gamma_u\right)
=0.
\end{equation}
By construction, all the curves $\Gamma_u(t,.)$ have the same base point $\Gamma_u(t,0)$ fixed at the origin. If they were smooth,
equation (\ref{Gamma_u_1}) would directly imply the existence of a  function $c_u=c_u(t)$ such that the function
$$\gamma_u(t,x)=\Gamma_u(t,x)+c_u(t)$$
is a solution to (\ref{FCB}) with $g=g(t,x)$. In this case, we have 
\begin{eqnarray*}
c_u(t)&=& \gamma_u(t,x)-\Gamma_u(t,x)\\
&=& \gamma_u(0,x)+\int_0^tg(\tau,x)u(\tau,x)\wedge \partial_xu(\tau,x)d\tau-\int_0^xu(t,z)dz\\
&=& \gamma_u(0,0)+\int_0^x\left(u(0,z)-u(t,z)\right)dz +\int_0^tg(\tau,x)u(\tau,x)\wedge \partial_xu(\tau,x)d\tau.
\end{eqnarray*}
Note that the function $c_u$ represents the evolution in time of the actual base point of the curves $\gamma_u(t,.)$. 

The relation between  the modified binormal curvature flow equation (\ref{FCB}) and the equation (\ref{LIA}) is specified in Lemma \ref{FCB_LIA},
stated above, whose proof we now present.
\subsubsection{Proof of Lemma \ref{FCB_LIA}}
Define $a\in \mathcal{D}'(I\times\mathbb{R},\mathbb{R}^3)$ by letting
 $$a(t,x) = \int_0^x\left(\omega(0,z)-\omega(t,z)\right)dz +\int_0^tg(\tau,x)\omega(\tau,x)\wedge\partial_x\omega(\tau,x)d\tau.$$
Let $\chi \in \mathcal{D}(\mathbb{R},\mathbb{R})$ be such that $\int_\mathbb{R}\chi(z)dz=1$ and set
$$c_\omega(t)=\int_\mathbb{R}\chi(z)a(t,z)dz.$$ 
By construction, we have $c_\omega(0)=0$ and, since  $\omega \in W^{1,\infty}(I,H^{-1}(\mathbb{R})),$ we also have $c_\omega\in \mathcal{C}(I,\mathbb{R}^3).$
On the other hand, we have
\begin{eqnarray}\label{Gamma_u_2}
\partial_x(\partial_ta) &=& \partial_t\partial_x\Gamma_\omega
-\partial_x\left(g\omega\wedge \partial_x\omega\right)\notag\\
&=&\partial_t\omega- \omega\wedge \Delta_g\omega \notag\\
&=& 0,
\end{eqnarray}
since $\omega$ is a solution to (\ref{LIA}). Hence $\partial_ta(t,x)$ does not depend on $x$ which implies in turn that,
for all $\varphi\in\mathcal{D}(I,\mathbb{R}^3)$, 
\begin{eqnarray}\label{Gamma_u_5}
\int_Ic_\omega(t)\cdot\varphi'(t)dt &=& \int_I\int_\mathbb{R}\chi(z)a(t,z)\cdot\varphi'(t)dtdz\notag\\
&=& -\int_\mathbb{R}\chi(z)\int_I\partial_ta(t,z)\cdot\varphi(t)dtdz\notag\\
&=&-\int_I\partial_ta(t,z)\cdot\varphi(t)dt.
\end{eqnarray}
Relation (\ref{Gamma_u_5}) means that 
\begin{equation}\label{gamma_u_6}
c_\omega'=\partial_t a=-\partial_t\Gamma+g\omega\wedge\partial_x\omega\quad \text{in}\quad \mathcal{D}'(I,\mathbb{R}^3).
\end{equation}
Now we show that the function $\gamma_\omega,$ defined on $I\times \mathbb{R}$ by
$$\gamma_\omega(t,x)=\Gamma_\omega(t,x)+c_\omega(t),$$
is a solution to (\ref{FCB}) on $I\times\mathbb{R}.$ To this end, let $\psi\in \mathcal{D}(I\times\mathbb{R},\mathbb{R}^3)$ and let
$$\varphi(t)=\int_{\mathbb{R}}\psi(t,z)dz \in \mathcal{D}(I,\mathbb{R}^3).$$
Then, in view of (\ref{gamma_u_6}), we have
\begin{eqnarray*}
\langle \partial_t\gamma_\omega
-g\partial_x\gamma_\omega\wedge\partial^2_x\gamma_\omega,\psi\rangle_{I\times\mathbb{R}}
&=& \langle \partial_t\Gamma_\omega-g\omega\wedge\partial_x\omega,\psi\rangle_{I\times\mathbb{R}}
+\langle c_\omega,\psi\rangle_{I\times\mathbb{R}}\\
&=& -\langle \partial_t a,\varphi \rangle_I+\langle c_\omega',\varphi\rangle_I\\
&=& 0,
\end{eqnarray*}
where $\langle,\rangle_{I\times\mathbb{R}}$ denotes the duality pairing between $ \mathcal{D}'(I\times\mathbb{R},\mathbb{R}^3)$ and
$ \mathcal{D}(I\times\mathbb{R},\mathbb{R}^3),$ while $\langle,\rangle_I$ denotes that between $ \mathcal{D}'(I,\mathbb{R}^3)$ and
$ \mathcal{D}(I,\mathbb{R}^3).$ This proves the existence of $c_\omega.$ The uniqueness of $c_\omega$ follows by noting that any other 
function $\tilde{c}_\omega$ is required to be continuous  with $\tilde{c}_\omega(0)=0$, and since 
its distributional derivative $\tilde{c}_\omega'=\partial_ta = c_\omega$.

\section{Approximation by discretization of the Schr\"{o}dinger map equation }
We present here the strategy of proof of theorems \ref{th_ex_np}, \ref{th_ex_p} and \ref{th_ex_npf}.
We discretise, in space, the continuous system 
\begin{equation}\label{PLIAg}
\left\{\begin{array}{lr}
\partial_t u=\partial_x\left(u\wedge g\partial_x u\right)=
u\wedge \partial_x\left(g\partial_x u\right),& t\geq 0,\quad x\in \mathbb{R},\\
u(0,.)=u_0,&
\end{array}\right.
\end{equation}
in the following sense:\\
For some $h>0,$ we consider the sequence $u_h\equiv\{u_h(t,x_i)\}_{i\in \mathbb{Z}}$ satisfying the semi-discrete system
\begin{equation}\label{PLIAgD}
\left\{\begin{array}{lr}
\frac{du_h}{dt}=D^+\left(u_h\wedge g_hD^-u_h\right)=
u_h\wedge D^+\left(g_hD^-u_h\right),& t\geq 0,\\
u_h(0,x_i)=u_h^0(x_i),\quad i\in\mathbb{Z}, &
\end{array}\right.
\end{equation}
where $\{x_i\}_{i\in \mathbb{Z}},$ is a uniform subdivision of $\mathbb{R}$ with step  $h,$  $g_h\equiv\{ g(t,x_i)\}_{i\in \mathbb{Z}},$ and 
$D^+, D^-$ are two operators approximating the derivative operator $\partial_x$. The sequence $\{ u_h^0(x_i)\}_{i\in \mathbb{Z}}$ is constructed 
such that it converges to $u_0$ in certain sense.
We solve the  problem (\ref{PLIAgD}) in some space  discretising  the space $L^\infty(\mathbb{R}^+,H^1_{loc}(\mathbb{R}))$
where our research for solving the continuous problem (\ref{PLIAg}) takes a place. Then, we prove boundedness properties for the discrete derivatives
$D^+u_h$, $D^-D^+u_h$ and $D^+D^-D^+u_h$. This allows us, 
by using the compactness properties in the spaces $L^2(\mathbb{R})$ and $H_{loc}^1(\mathbb{R}),$ to extract a subsequence $\{u_h\}_h$
\footnote{To give sense to the notation $\{u_h\}_h$, we can consider $h:\mathbb{N}\rightarrow \mathbb{R}^+$ to be a strictly decreasing function  
which  to zero when $n\rightarrow +\infty.$ We have made this choice for its simplicity.}
which converges to a solution of (\ref{PLIAg}).
The proof of Theorem \ref{th_un} is standard. We consider two solutions $u$ and $\tilde{u}$ with initial datum
$u_0$ and $\tilde{u}_0$ respectively, then we prove Gr\"{o}nwall-type inequalities for $\|u-\tilde{u}\|_{H^1}$ and $\|u-\tilde{u}\|_{H^2}$.
For  Theorem \ref{ex_fcb_g(gamma)}, we follow the same strategy followed in the proof of Theorem \ref{th_ex_npf}.

In what follows, we define the elements of the discrete problem (\ref{PLIAgD}). Then, we prove some convergence  properties before we   
skip to the proofs of main theorems.

\begin{definition}
Let $h>0.$ Let
$$\mathbb{Z}_h=\{x_i\in \mathbb{R},\quad x_{i+1}-x_i = h \quad\forall i\in \mathbb{Z}\}.$$
We define the two spaces $L_h^2$ and $L_h^\infty$ by
$$L_h^2=\{v_h=\{v_h(x_i)\}_i\in (\mathbb{R}^3)^{\mathbb{Z}_h},\quad \sum_i|v_h(x_i)|^2<+\infty \},$$
$$L_h^\infty=\{v_h=\{v_h(x_i)\}_i\in (\mathbb{R}^3)^{\mathbb{Z}_h},\quad \sup_i|v_h(x_i)|<+\infty \}.$$
We define the scalar product  $(,)_h$ on $L_h^2$ by
$$(u_h,v_h)_h=h\sum_iv_h(x_i)\cdot u_h(x_i),\quad u_h, v_h \in L_h^2.$$
Its associated norm $|.|_h$ is defined by 
$$|v_h|_h^2=h\sum_i|v_h(x_i)|^2.$$
Let $l>0,$  $N\in \mathbb{N}$ and $h=\frac{l}{N}.$ We define the space of $N$-periodic sequences 
$$P_{l,N}=\{v_h \in (\mathbb{R}^3)^{\mathbb{Z}_h},\quad v_h(x_i)=v_h(x_{i+N}),\quad i\in\mathbb{Z}\}.$$
We define the scalar product $(,)_{l,N}$ by 
$$(u_h,v_h)_{l,N}=h\sum_{i=1}^{N}v_h(x_i)\cdot u_h(x_i).$$
Its associated norm $|.|_{l,N}$ is defined by 
$$|v_h|_{l,N}^2=h\sum_{i=1}^{i=N}|v_h(x_i)|^2.$$
Let $v_h\in (\mathbb{R}^3)^{\mathbb{Z}_h}$. We define the left and the right approximations of the derivatives in $x_i$ by 
\begin{equation*}
\left\{\begin{array}{lr}
D^-v_h(x_i)=\frac{v_h(x_i)-v_h(x_{i-1})}{h},&\\
D^+v_h(x_i)=\frac{v_h(x_{i+1})-v_h(x_i)}{h}.&
\end{array}\right.
\end{equation*}
It is clear that for two sequences $u_h=\{u_h(x_i)\}_i$ and $v_h=\{v_h(x_i)\}_i$,  we have 
$$D^\pm(u_hv_h)=\tau^\pm u_hD^\pm v_h+ D^\pm u_h v_h,$$ with
$$\tau^\pm u_h(x_i) = u_h(x_{i\pm1}).$$
\end{definition}
The two spaces  $L^2_h$ and $P_{l,N}$ verify the following property  
\begin{lemma}\label{Pd}
\begin{enumerate}
 \item  Let $v_h\in L_h^2,$ then  $D^+v_h\in L_h^2$ and we have 
$$|D^+v_h|_h\leq \frac{2}{h}|v_h|_h.$$
\item  Let $v_h\in P_{l,N},$ then  $D^+v_h\in P_{l,N}$  and we have 
$$|D^+v_h|_{l,N}\leq \frac{2}{h}|v_h|_{l,N},\quad h = \frac{l}{N}.$$
\end{enumerate}

\end{lemma}
\begin{proof}
It follows directly from the inequality  
$$|D^+_hv_h(x_i)|^2 \leq \frac{2}{h^2}(|v_h(x_i)|^2+|v_h(x_{i+1})|^2).$$
\end{proof}

\begin{definition}
We define the norm 
$$|v_h|^2_{H_h^1}=|v_h|^2_h+|D^+v_h|^2_h, \quad v_h \in L_h^2,$$
and the space 
$$H_h^{-1}=\left\{v_h\in (\mathbb{R}^3)^{\mathbb{Z}_h},\quad \sup_{u_h\in L_h^2}\frac{\langle v_h,u_h\rangle_h}{|u_h|_{H_h^1}}<+\infty \right\}.$$
It is clear that $L_h^2\subset H_h^{-1}$ and the function
$v_h\mapsto |v_h|_{H_h^{-1}}\equiv\sup_{u_h\in L_h^2}\frac{\langle v_h,u_h\rangle_h}{|u_h|_{H_h^1}}$ define a norm on $H_h^{-1}.$
Similarly, we define the norms 
$$|v_h|^2_{H_{l,N}^1}=|v_h|^2_{l,N}+|D^+v_h|^2_{l,N},$$
$$|v_h|_{H_{l,N}^{-1}}=\sup_{u_h\in P_{l,N}}\frac{\langle v_h,u_h\rangle_{l,N}}{|u_h|_{H_{l,N}^1}},\quad v_h\in P_{l,N}.$$
The two norms $|.|_{H_h^{-1}}$ and $|.|_{H_{l,N}^{-1}}$ are the dual norms of $|.|_{H_h^1}$ and $|.|_{H_{l,N}^1}$   
with respect to scalar product $\langle,\rangle_h$ et $\langle,\rangle_{l,N}$ respectively.
\end{definition}

\begin{lemma}[discrete integration by parts formula]
For all $(v_h,u_h) \in L_h^\infty\times\in L^2_h,$  we have 
\begin{equation}\label{IPPd}
\sum_iv_h(x_i)\cdot D^+u_h(x_i)=-\sum_iu_h(x_i)\cdot D^-v_h(x_i).
\end{equation}
Similarly, for all $v_h,u_h\in P_{l,N},$ we have  
\begin{equation}\label{IPPdp}
\sum_{i=1}^{N}v_h(x_i)\cdot D^+u_h(x_i)=-\sum_{i=1}^{N}u_h(x_i)\cdot D^-v_h(x_i).
\end{equation}
\end{lemma}
\begin{proof}
Let $v_h \in L_h^\infty,$ $u_h\in L^2_h$ and $K\in \mathbb{N}.$ We develop the sum $\sum_{i=-K}^{K}v_h(x_i)\cdot D^+u_h(x_i)$ and 
we make a change in index, then (\ref{IPPd}) holds by using the property 
$\lim_{|i|\rightarrow +\infty}|u_h(x_i)|=0$ and the assumption  ($v_h \in L_h^\infty$).  
In the second case, we simply develop the sum $\sum_{i=1}^{N}v_h(x_i)\cdot D^+u_h(x_i)$ and 
make a change in index, then we use the periodicity of $v_h$ and $u_h.$
\end{proof}

\begin{definition}
Let $h>0.$  We denote $C_i=[x_i,x_{i+1}[, i\in \mathbb{Z}.$ Let $P_h$ and $Q_h$ be the two interpolation operators 
defined, for all $v_h=\{v_h(x_i)\}_i\in (\mathbb{R}^3)^{\mathbb{Z}_h}$, by  
$$Q_hv_h:\quad \mathbb{R}\rightarrow \mathbb{R}^3,\quad x\mapsto Q_hv_h(x)=v_h(x_i),\quad \forall x\in C_i,\forall i\in \mathbb{Z},$$
$$P_hv_h:\quad \mathbb{R}\rightarrow \mathbb{R}^3,\quad x\mapsto P_hv_h(x)=v_h(x_i)+D^+v_h(x_i)(x-x_i),\quad \forall x\in C_i,\forall i\in \mathbb{Z}.$$
\end{definition}
In all what follows we keep the notation of this definition. We have the following important lemma 
\begin{lemma}\label{lem_born}
\begin{enumerate}
 \item  Let $\{v_h\}_h \in (\R^3)^{\Z_h}$ be such that
\begin{equation*}\left\{\begin{array}{lr}
v_h\in H^{-1}_h,\quad \forall h>0,&\\
\exists C>0,\quad |v_h|_{H^{-1}_h} < C.
\end{array}\right.\end{equation*}
Then the sequence $\{P_hv_h\}_h$ is bounded in $H^{-1}(\mathbb{R}).$
\item  Let $l>0$ and $\{v_h\}_h$ be a sequence satisfying
\begin{equation*}\left\{\begin{array}{lr}
h=\frac{l}{N},&\\
v_h\in P_{l,N},\quad \forall N \in \mathbb{N},&\\
\exists C>0,\quad |v_h|_{H^{-1}_{l,N}} < C,\quad \forall N \in \mathbb{N}.&
\end{array}\right.\end{equation*}
Then the sequence $\{P_hv_h\}_h$ is bounded in $H^{-1}(\mathbb{T}^l).$
\end{enumerate}

\end{lemma}
\begin{proof}
 We have 
\begin{eqnarray}\label{norm_du}
\|P_hv_h\|_{H^{-1}(\mathbb{R})}&=&\sup_{\varphi\in\mathcal{D}(\mathbb{R})}\frac{\langle P_hv_h,\varphi\rangle_{L^2(\mathbb{R})}}{\|\varphi\|_{H^1(\mathbb{R})}}\notag\\
&\leq& \sup_{\varphi\in\mathcal{D}(\mathbb{R})}\frac{\langle P_hv_h,P_h\varphi_h\rangle_{L^2(\mathbb{R})}}{\|\varphi\|_{H^1(\mathbb{R})}}+
\sup_{\varphi\in\mathcal{D}(\mathbb{R})}\frac{\langle P_hv_h,\varphi-P_h\varphi_h\rangle_{L^2(\mathbb{R})}}{\|\varphi\|_{H^1(\mathbb{R})}},
\end{eqnarray}
with $\varphi_h=\{\varphi(x_i)\}_i.$ Since 
$$\|\varphi-P_h\varphi_h\|_{L^2(\mathbb{R})}\leq h\|(\varphi-P_h\varphi_h)'\|_{L^2(\mathbb{R})}\quad (\text{Poincar\'{e}}),$$
we have 
\begin{eqnarray*}
\|\varphi\|^2_{H^1(\mathbb{R})}&=&\|P_h\varphi_h\|^2_{H^1(\mathbb{R})}+\|\varphi-P_h\varphi_h\|^2_{H^1(\mathbb{R})}+
2\int_{\mathbb{R}}(P_h\varphi_h).(\varphi-P_h\varphi_h)dx+2\int_{\mathbb{R}}(P_h\varphi_h)'.(\varphi-P_h\varphi_h)'dx\\
&\geq& \|P_h\varphi_h\|^2_{H^1(\mathbb{R})}+\|\varphi-P_h\varphi_h\|^2_{H^1(\mathbb{R})}-
2h\|P_h\varphi_h\|_{L^2(\mathbb{R})}\|(\varphi-P_h\varphi_h)'\|_{L^2(\mathbb{R})}.
\end{eqnarray*}
Then there exists $h_0>0$ such that for all $h<h_0,$ we have 
\begin{eqnarray*}
\|\varphi\|^2_{H^1(\mathbb{R})}&\geq& \frac{1}{2}\left(\|P_h\varphi_h\|^2_{H^1(\mathbb{R})}+\|\varphi-P_h\varphi_h\|^2_{H^1(\mathbb{R})}\right)\\
&\geq& \frac{1}{2}\max\left(\|P_h\varphi_h\|^2_{H^1(\mathbb{R})}, \|\varphi-P_h\varphi_h\|^2_{H^1(\mathbb{R})}\right).
\end{eqnarray*}
We obtain by substituting in (\ref{norm_du}) 
\begin{equation}\label{norm_du1}
\|P_hv_h\|_{H^{-1}(\mathbb{R})}\leq \sup_{\varphi\in\mathcal{D}(\mathbb{R})}\frac{\langle P_hv_h,P_h\varphi_h\rangle_{L^2(\mathbb{R})}} {\frac{1}{\sqrt{2}}\|P_h\varphi_h\|_{H^1(\mathbb{R})}}
+\sqrt{2}h\|P_hv_h\|_{L^2(\mathbb{R})}.
\end{equation}
Next, we have 
\begin{eqnarray*}
\|P_h\varphi_h\|^2_{H^1(\mathbb{R})}&=&\sum_i\int_{x_i}^{x_{i+1}}\left|\frac{x_i-x}{h}\varphi(x_i)+\frac{x-x_i}{h}\varphi(x_{i+1})\right|^2dx+
\sum_ih\left|\frac{\varphi(x_i)-\varphi(x_{i+1})}{h}\right|^2dx\\
&=&\sum_i\frac{h}{3}\left(|\varphi(x_i)|^2+|\varphi(x_{i+1})|^2+\varphi(x_{i+1})\varphi(x_i)\right)+|D^+\varphi_h|^2_h\\
&\geq& \sum_i\frac{h}{6}\left(|\varphi(x_i)|^2+|\varphi(x_{i+1})|^2\right)+|D^+\varphi_h|^2_h,
\end{eqnarray*}
from which we can write
\begin{equation}\label{norm_du2}
\|P_h\varphi_h\|^2_{H^1(\mathbb{R})}\geq\frac{1}{3}|\varphi_h|^2_h+|D^+\varphi_h|^2_h\geq \frac{1}{3}|\varphi_h|^2_{H^1_h}.
\end{equation}
We have on the one hand
\begin{eqnarray}\label{norm_du3}
\langle P_hv_h,P_h\varphi_h\rangle_{L^2(\mathbb{R})}&=&\sum_i\int_{x_i}^{x_{i+1}}(v_h(x_i)+D^+v_h(x_i)(x-x_i)).(\varphi(x_i)+D^+\varphi_h(x_i)(x-x_i))dx \notag\\
&=&(v_h,\varphi_h)_h+\frac{h}{2}(v_h,D^+\varphi_h)_h+\frac{h}{2}(D^+v_h,\varphi_h)_h+\frac{h^2}{3}(D^+v_h,D^+\varphi_h)_h\notag\\
&=&(v_h,\varphi_h)_h+\frac{h}{2}(v_h,D^+\varphi_h)_h-\frac{h}{2}(v_h,D^-\varphi_h)_h+\frac{h^2}{3}(D^+v_h,D^+\varphi_h)_h\notag\\
&\leq& (v_h,\varphi_h)_h+h|v_h|_h|D^+\varphi_h|_h+\frac{h^2}{3}|D^+v_h|_h|D^+\varphi_h|_h,
\end{eqnarray}
and on the other hand
\begin{eqnarray}\label{norm_du4}
\|P_hv_h\|^2_{L^2(\mathbb{R})}&=&\sum_i\int_{x_i}^{x_{i+1}}|v_h(x_i)+D^+v_h(x_i)(x-x_i)|^2dx\notag\\
&\leq&2\sum_i\int_{x_i}^{x_{i+1}}(|v_h(x_i)|^2+|D^+v_h(x_i)|^2(x-x_i)^2)dx\notag\\
&=&2|v_h|^2_h+\frac{2h^2}{3}|D^+v_h|^2_h.
\end{eqnarray}
Then by combining (\ref{norm_du1}), (\ref{norm_du2}), (\ref{norm_du3}) and (\ref{norm_du4}) we get 
\begin{eqnarray*}
\|P_hv_h\|_{H^{-1}(\mathbb{R})}&\leq&
\sup_{\varphi\in\mathcal{D}(\mathbb{R})}\frac{(v_h,\varphi_h)_h+h|v_h|_h|D^+\varphi_h|_h+\frac{h^2}{3}|D^+v_h|_h|D^+\varphi_h|_h}{\frac{1}{\sqrt{6}}|\varphi_h|_{H^1_h}}+
2h|v_h|_h+\frac{2h^2}{\sqrt{3}}|D^+v_h|_h\\
&\leq&\sqrt{6}(|v_h|_{H^{-1}_h}+h|v_h|_h+\frac{h^2}{3}|D^+v_h|_h)+2h|v_h|_h+\frac{2h^2}{\sqrt{3}}|D^+v_h|_h\\
&\leq& \sqrt{6}|v_h|_{H^{-1}_h}+(\sqrt{6}+2)h|v_h|_h+2(\frac{\sqrt{6}}{3}+\frac{2}{\sqrt{3}})h|v_h|_h \footnote{On a utilis\'{e} (\ref{Pd})} \\
&\leq& \sqrt{6}|v_h|_{H^{-1}_h}+(\sqrt{6}+2+2\frac{2+\sqrt{2}}{\sqrt{3}})h|v_h|_h\\
&\leq& C|v_h|_{H^{-1}_h},
\end{eqnarray*}
since
\begin{eqnarray*}
|v_h|_{H_h^{-1}}&=&\sup_{u_h}\frac{(v_h,u_h)_h}{|u_h|_{H_h^1}}\\
&\geq&\sup_{u_h}\frac{(v_h,u_h)_h}{[|u_h|^2_h+\frac{4}{h^2}|u_h|^2_h]^{\frac{1}{2}}}\\
&=&\frac{h}{\sqrt{h^2+4}}\sup_{u_h}\frac{( v_h,u_h)_h}{|u_h|_h}\\
&\geq&\frac{h}{\sqrt{h_0^2+4}}|v_h|_h,\quad \forall h\geq h_0.
\end{eqnarray*}
The  proof of 2 is similar to that of 1.
\end{proof}

The following lemma shows that the space $L_h^2,$ equipped with the norm $|.|_{H_h^1},$  is continuously embedded in $L_h^\infty.$
\begin{lemma}\label{equiv_norms}
There exist two constants $C_1, C_2>0$ such that for all $h>0$ and  $v_h\in L_h^2$ we have 
$$C_2|v_h|_h\leq \|P_hv_h\|_{L^2(\mathbb{R})}\leq C_1|v_h|_h.$$
\end{lemma}
\begin{proof}
Since  
\begin{eqnarray*}
\int_{x_i}^{x_{i+1}}|u_h(x_i)+D^+u_h(x_i)(x-x_i)|^2dx &=&h|u_h(x_i)|^2+\frac{1}{2}h^2u_h(x_i)\cdot D^+u_h(x_i)+\frac{1}{3}h^3|D^+u_h(x_i)|^2\\
&=& \frac{5}{6}h|u_h(x_i)|^2- \frac{1}{6}hu_h(x_i)\cdot D^+u_h(x_i)+\frac{1}{3}h|u_h(x_{i+1})|^2,
\end{eqnarray*}
and 
$$\frac{3}{4}|u_h(x_i)|^2+\frac{1}{4}|u_h(x_{i+1})|^2\leq\frac{5}{6}|u_h(x_i)|^2- \frac{1}{6}u_h(x_i) \cdot D^+u_h(x_i)+\frac{1}{3}|u_h(x_{i+1})|^2\leq \frac{11}{12}|u_h(x_i)|^2+\frac{5}{12}|u_h(x_{i+1})|^2,$$
we have 
$$|v_h|_h^2\leq \|P_hv_h\|^2_{L^2(\mathbb{R})}\leq \frac{4}{3}|v_h|_h^2.$$
\end{proof}

\begin{corollary}\label{Sobolev_discret}
If $v_h\in L_h^2\subset L_h^\infty,$ then $P_hv_h\in H^1(\mathbb{R})$ and there exists $C>0$ (which does not depend on $h$) such that 
$$|v_h|_{L_h^\infty}\leq C|v_h|_{H_h^1}.$$
\end{corollary}
\begin{proof}
Since $\frac{dP_hv_h}{dx}=Q_hD^+v_h \in L^2(\mathbb{R}),$  we have $P_hv_h\in H^1(\mathbb{R})$ (Lemma \ref{Pd}). On the other hand, we have 
\begin{eqnarray*}
\|P_hv_h\|_{L^\infty(\mathbb{R})}&=&\sup_{i\in \mathbb{Z}}\sup_{x\in [x_i, x_{i+1}[}|u_h(x_i)+D^+u_h(x_i)(x-x_i)|\\
&=&\sup_{i\in \mathbb{Z}}\max(|u_h(x_i)|,|u_h(x_{i+1})|)\\
&=& |v_h|_{L_h^\infty}.
\end{eqnarray*}
The space $L^\infty(\mathbb{R})$ is continuously embedded in the space $H^1(\mathbb{R})$ (Sobolev) and there exists $\tilde{C}>0$ such that 
$$\|v\|_{L^\infty(\mathbb{R})}\leq \tilde{C}\|v\|_{H^1(\mathbb{R})},\quad \forall v\in H^1(\mathbb{R}).$$
Consequently, 
\begin{eqnarray*}
|v_h|^2_{L_h^\infty}&=&\|P_hv_h\|^2_{L^\infty(\mathbb{R})}\\
&\leq& \tilde{C}^2 \|P_hv_h\|^2_{H^1(\mathbb{R})}\\
&\leq& \tilde{C}^2(C_1^2|v_h|_h^2+\|Q_hD^+v_h\|^2_{L^2(\mathbb{R})})\\
&\leq&  C^2|v_h|^2_{H_h^1}.
\end{eqnarray*}
\end{proof}

\section{Proofs of main theorems}
Let us first show some important properties. 
\subsection{Convergence properties}
\begin{lemma}\label{lem_cv_l2f}
\begin{enumerate}
 \item Let $\{v_h\}_h \in (\R^3)^{\Z_h}$ be such that   
$$v_h\in L^2_h,\quad \forall h,$$ and 
\begin{equation}\label{suit_borne}
\exists C>0,\quad |v_h|_h\leq C.
\end{equation}
Then the sequence  $\{P_hv_h-Q_hv_h\}_h$ converges weakly to zero in $L^2(\mathbb{R}).$
\item Let $l>0$ and $\{v_h\}_h$ be a sequence satisfying   
\begin{equation*}\left\{\begin{array}{lr}
h=\frac{l}{N},&\\
v_h\in P_{l,N},\quad \forall N\in \mathbb{N},&
\end{array}\right.\end{equation*}
and 
\begin{equation}
\exists C>0,\quad |v_h|_{l,N}\leq C.
\end{equation}
Then $\{P_hv_h-Q_hv_h\}_h$ converges weakly to zero in $L^2(\mathbb{T}^l).$ 
\end{enumerate}
Moreover, if $\{Q_hv_h\}_h$ converges to $v$ in $L^2$ ($L^2(\mathbb{R})$ or $L^2(\mathbb{T}^l)$), then
$\{P_hv_h\}_h$ converges to the same limit in $L^2.$
\end{lemma}
\begin{proof}
We write
\begin{eqnarray}\label{a}
\|P_hv_h-Q_hv_h\|^2_{L^2(\mathbb{R})}&=&\sum_i\int_{x_i}^{x_{i+1}}|D^+v_h(x_i)|^2(x-x_i)^2dx\notag\\
&\leq&\frac{1}{3}h^3\sum_i|D^+v_h(x_i)|^2\notag\\
&=&\frac{1}{3}h^2|D^+v_h|^2_h\notag\\
&\leq& \frac{4}{3}|v_h|^2_h\notag\\
&\leq& \frac{4}{3}C^2.
\end{eqnarray}
Furthermore, for all $\varphi\in \mathcal{D}(\mathbb{R}),$ we have 
\begin{eqnarray}\label{b}
|\langle P_hv_h-Q_hv_h,\varphi\rangle_{L^2(\mathbb{R})}| &\leq& |\langle P_hv_h-Q_hv_h,Q_h\varphi\rangle_{L^2(\mathbb{R})}|\notag\\
                                                         &&+ \|P_hv_h-Q_hv_h\|_{L^2(\mathbb{R})}\|\varphi-Q_h\varphi_h\|_{L^2(\mathbb{R})},
\end{eqnarray}
where $\varphi_h=\{\varphi_(x_i)\}_i.$ We have on the one hand
\begin{eqnarray}\label{c}
\|\varphi-Q_h\varphi_h\|^2_{L^2(\mathbb{R})}&=&\sum_i\int_{x_i}^{x_{i+1}}|\varphi(x)-\varphi(x_i)|^2dx\notag\\
&=& \sum_i\int_{x_i}^{x_{i+1}}|\int_{x_i}^{x}\varphi'(s)ds|^2dx\notag\\
&\leq& \sum_i\int_{x_i}^{x_{i+1}}(\int_{x_i}^{x}|\varphi'(s)|^2ds)(x-x_i)dx\notag\\
&\leq& \frac{h^2}{2}\sum_i\int_{x_i}^{x}|\varphi'(s)|^2ds\notag\\
&=& \frac{h^2}{2}\|\varphi'\|^2_{L^2(\mathbb{R})}.
\end{eqnarray}
On the other hand, we have 
\begin{eqnarray}\label{d}
|\langle P_hv_h-Q_hv_h,Q_h\varphi\rangle_{L^2(\mathbb{R})}|&=& \frac{h}{2}|\langle D^+v_h,\varphi_h\rangle_h|\notag\\
&=&\frac{h}{2}|\langle v_h,D^-\varphi_h\rangle_h|\notag\\
&\leq& \frac{h}{2}|v_h|_h|D^-\varphi_h|_h\notag\\
&\leq& \frac{1}{2}C\left[h\sum_i|\int_{x_{i-1}}^{x_i}\varphi'(s)ds|^2\right]^{\frac{1}{2}}\notag\\
&\leq& \frac{1}{2}Ch\|\varphi'\|_{L^2(\mathbb{R})}.
\end{eqnarray}
Then combining (\ref{a}), (\ref{b}), (\ref{c}) and (\ref{d}), we obtain
$$|\langle P_hv_h-Q_hv_h,\varphi\rangle_{L^2(\mathbb{R})}|\leq (\frac{2}{\sqrt{6}}+\frac{1}{2})C\|\varphi'\|_{L^2(\mathbb{R})}h.$$
Thus the proof of 1. is completed. The proof of 2 is similar to that of 1. To prove the strong convergence property, let
$v\in L^2,$ then it suffices to note that  
\begin{eqnarray*}
\|P_hv_h-Q_hv_h\|^2_{L^2}&=&\sum_i\int_{x_i}^{x_{i+1}}|D^+v_h(x_i)|^2(x-x_i)^2dx\notag\\
&=&\frac{1}{3}h^3\sum_i|D^+v_h(x_i)|^2\notag\\
&=&\frac{1}{3}\|\tau_{-h}Q_hv_h-Q_hv_h\|^2_{L^2}\notag,\\
\end{eqnarray*}
with $\tau_{h}w = w(\cdot-h),$ and  
\begin{eqnarray*}
\|\tau_{-h}Q_hv_h-Q_hv_h\|_{L^2}&\leq& \|\tau_{-h}Q_hv_h-v\|_{L^2} +\|Q_hv_h-v\|_{L^2}\\
&\leq& \|\tau_hv-v\|_{L^2} +2\|Q_hv_h-v\|_{L^2}.\\
\end{eqnarray*}
Thus the fact that $\lim_{h\rightarrow0}\|\tau_hv-v\|_{L^2} = 0$ completes the proof. 
\end{proof}

\begin{lemma}\label{lem_cv_h-1}
\begin{enumerate}
 \item  Let $v \in H^{-1}(\mathbb{R}),$ and $\{v_h\}_h\in (\R^3)^{\Z_h}$ be  such that the sequence $\{Q_hv_h\}_h$ converges to $v$ in $H^{-1}(\mathbb{R})$ 
weak star. Then the sequence $\{P_hv_h\}_h$ converges to $v$ in $H^{-1}(\mathbb{R})$ weak star.
\item Let $l>0,$ $v^l \in H^{-1}(\mathbb{T}^l)$ and $\{v_h\}_h$ be a sequence satisfying
\begin{equation*}\left\{\begin{array}{lr}
h=\frac{l}{N},&\\
v_h\in l_N,\quad \forall N\in \mathbb{N},&\\
Q_hv_h\rightarrow v^l \quad \text{in}\quad H^{-1}(\mathbb{T}^l)\quad \text{weak star}.&
\end{array}\right.\end{equation*}
Then $\{P_h^lv_h\}_h$ converges to $v^l$ in $H^{-1}(\mathbb{T}^l)$ weak star.
\end{enumerate}
\end{lemma}                                                                                          
\begin{proof}
First, we prove that  $P_hv_h\in H^{-1}(\mathbb{R}), \forall h.$ To this end, we first write 
$$P_hv_h=Q_hv_h+(P_h-Q_h)v_h.$$
Then, it suffices to prove that $(P_h-Q_h)v_h\in H^{-1}(\mathbb{R}), \forall h.$ Let $\varphi\in\mathcal{D}(\mathbb{R}),$ and 
$\varphi_h=\{\varphi(x_i)\}_i.$ We have 
\begin{eqnarray*}
|\langle P_hv_h-Q_hv_h,\varphi\rangle_{L^2(\mathbb{R})}|&\leq &|\langle P_hv_h-Q_hv_h,Q_h\varphi\rangle_{L^2(\mathbb{R})}|+
|\langle P_hv_h-Q_hv_h,\varphi-Q_h\varphi\rangle_{L^2(\mathbb{R})}|\\
&\leq& \frac{h}{2}|(D^+v_h,\varphi_h)_h|+|\sum_i\int_{x_i}^{x_{i+1}}(D^+v_h(x_i).\int_{x_i}^{x}\varphi'(s)ds)(x-x_i)dx|\\
&\leq& \frac{h}{2}|(v_h,D^-\varphi_h)_h|+|h^2\sqrt{h}\sum_i|D^+v_h(x_i)|.\int_{x_i}^{x_{i+1}}|\varphi'(x)|^2dx|\\
&\leq& \frac{h}{2}|v_h|_h|D^-\varphi_h|_h+h^2|D^+v_h|_h\|\varphi'\|_{L^2(\mathbb{R})}\\
&\leq& \frac{h}{2}|v_h|_h\|\varphi'\|_{L^2(\mathbb{R})}+2h|v_h|_h\|\varphi'\|_{L^2(\mathbb{R})}\\
&\leq& \frac{5}{2}h|v_h|_h\|\varphi'\|_{L^2(\mathbb{R})},
\end{eqnarray*}
where the sequence $\{h|v_h|_h\}_h$ is bounded. Indeed, the sequence $\{Q_hv_h\}_h$ converges to $v$ in $H^{-1}(\mathbb{R})$ weak
star. Then there exists $C>0$ such that $\|Q_hv_h\|_{H^{-1}(\mathbb{R})}\leq C$ for all  $h,$ hence we have
\begin{equation}
\frac{\langle Q_hv_h,R_h^Nv_h\rangle_{L^2(\mathbb{R})}}{\|R_h^Nv_h\|_{H^1(\mathbb{R})}}\leq C,\quad \forall h,\forall N \in \mathbb{N},
\end{equation}
where $R_h^Nv_h$ is a piecwise function with compact support (hence $R_h^Nv_h\in H^1(\mathbb{R}) $) such that 
\begin{equation}\left\{\begin{array}{lr}
\langle Q_hv_h,R_h^Nv_h\rangle_{L^2(\mathbb{R})}=h\sum_{-N}^N|v_i|^2,&\\
\|R_h^Nv_h\|^2_{H^1(\mathbb{R})} \leq h^{-1}\sum_{-N}^N|v_i|^2.&
\end{array}\right.\end{equation}
For example, we can take $$R^N_hv_h=Q_h\tilde{v}_h+\sum_iD^+\tilde{v}_h\chi(x-x_i),$$
where  $\tilde{v}_h=\{\tilde{v}_h(x_i)\}_i$ with
$\tilde{v}_h(x_i)=\left\{\begin{array}{lr}v_h(x_i),& |i|\leq N\\ 0,& |i|> N, \end{array}\right.$ and $\chi$ is given by

\begin{equation*}\chi(x)=\left\{\begin{array}{lr}
0,\quad x<0 \quad \text{or} \quad x>h&\\
-\frac{3}{2}x,\quad x\in [0,\frac{h}{3}[&\\
\frac{3}{2}(x-\frac{2h}{3}),\quad x\in [\frac{h}{3},\frac{2h}{3}[&\\
3(x-\frac{2h}{3}),\quad x\in [\frac{2h}{3},h[.&\\
\end{array}\right.\end{equation*}
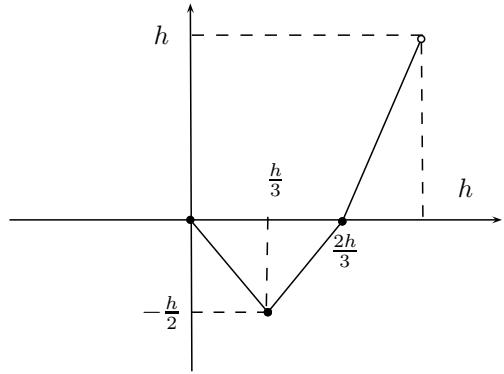
\begin{figure}[h]
 \centering
\scalebox{1} 
{
\begin{pspicture}(0,-2.44)(6.49,2.44)
\psline[linewidth=0.02cm,arrowsize=0.05291667cm 2.0,arrowlength=1.4,arrowinset=0.4]{->}(0.0,-0.43)(6.48,-0.43)
\psline[linewidth=0.02cm,arrowsize=0.05291667cm 2.0,arrowlength=1.4,arrowinset=0.4]{->}(2.4,-2.43)(2.38,2.43)
\psline[linewidth=0.02cm,dotsize=0.07055555cm 2.0]{*-o}(2.38,-0.43)(3.4,-1.65)
\psline[linewidth=0.02cm,dotsize=0.07055555cm 2.0]{*-o}(3.4,-1.65)(4.38,-0.45) 
\psline[linewidth=0.02cm,dotsize=0.07055555cm 2.0]{*-oo}(4.38,-0.45)(5.44,2.01)
\psline[linewidth=0.02cm,linestyle=dashed,dash=0.16cm 0.16cm](3.38,-1.57)(3.4,-0.39)
\psline[linewidth=0.02cm,linestyle=dashed,dash=0.16cm 0.16cm](5.42,1.91)(5.44,-0.39)
\psline[linewidth=0.02cm,linestyle=dashed,dash=0.16cm 0.16cm](5.44,2.01)(2.4,2.01)
\psline[linewidth=0.02cm,linestyle=dashed,dash=0.16cm 0.16cm](3.4,-1.65)(2.4,-1.65)
\usefont{T1}{ptm}{m}{n}
\rput(3.5,0.1){$\frac{h}{3}$}
\usefont{T1}{ptm}{m}{n}
\rput(4.414531,-0.84){$\frac{2h}{3}$}
\usefont{T1}{ptm}{m}{n}
\rput(6,0.0){$h$}
\usefont{T1}{ptm}{m}{n}
\rput(2,2.01){$h$}
\rput(2,-1.65){$-\frac{h}{2}$}
\usefont{T1}{ptm}{m}{n}
\end{pspicture} 
}
\caption{The function $\chi.$}
\end{figure}

Since $$\frac{h\sum_N^N|v_i|^2}{\left[h^{-1}\sum_{-N}^N|v_i|^2\right]^{\frac{1}{2}}}\leq C,\quad \forall h,\forall N \in \mathbb{N},$$
we get $h|v_h|_h\leq C, \forall h.$ Finally, we have $\|(P_h-Q_h)v_h\|_{H^{-1}(\mathbb{R})}\leq C, \forall h,$ then \\$\|P_hv_h\|_{H^{-1}(\mathbb{R})}\leq C, \forall h.$

To show that $\{P_hv_h\}_h$ converges to $v$ in $H^{-1}(\mathbb{R})$ weak star, we need to prove that 
$$P_hv_h \rightharpoonup v,\quad \text{in}\quad \mathcal{D}'(\mathbb{R}).$$
To this end, let $\varphi \in \mathcal{D}(\mathbb{R}).$ We denote $\tau_h\varphi=\frac{1}{2}(\varphi + \varphi(.-h)).$ Then we have
\begin{eqnarray*}
\langle P_hv_h,\varphi\rangle_{L^2(\mathbb{R})}&=& \sum_i\int_{x_i}^{x_{i+1}}(v_h(x_i)+D^+v_h(x_i)(x-x_i))\cdot\varphi(x)dx\\
&=& \sum_i\int_{x_i}^{x_{i+1}}\left(\frac{v_h(x_i)+v_h(x_{i+1})}{2}+D^+v_h(x_i)(x-x_i-\frac{h}{2})\right)\cdot\varphi(x)dx\\
&=& \langle Q_hv_v,\tau_h\varphi\rangle_{L^2(\mathbb{R})}+\sum_i\int_{x_i}^{x_{i+1}}\left(D^+v_h(x_i)(x-x_i-\frac{h}{2})\cdot\int_{x_i}^{x}\varphi(t)dt\right)dx\\
&=& \langle Q_hv_v,\tau_h\varphi\rangle_{L^2(\mathbb{R})}+\int_0^h\int_0^s(s-\frac{h}{2})\left(\sum_iD^+v_h(x_i)\cdot\varphi'(x_i+\rho)\right)d\rho ds \\
&=& \langle Q_hv_v,\tau_h\varphi\rangle_{L^2(\mathbb{R})}
+\frac{1}{h}\int_0^h\int_0^s(s-\frac{h}{2})\left(\sum_iv_h(x_i)\cdot(\varphi'(x_{i-1}+\rho)-\varphi'(x_i+\rho))\right)d\rho ds \\
&=& \langle Q_hv_v,\tau_h\varphi\rangle_{L^2(\mathbb{R})}
+\frac{1}{h}\int_0^h\int_0^s(s-\frac{h}{2})\left(\sum_iv_h(x_i)\cdot\int_{x_i}^{x_{i+1}}\varphi''(x+\rho)dx\right)d\rho ds, \\
\end{eqnarray*}
where
\begin{equation}\left\{\begin{array}{lr}
Q_hv_h\rightarrow v ,\quad\text{in}\quad H^{-1}(\mathbb{R})\quad \text{weak star},&\\
\tau_h\varphi \rightarrow \varphi,\quad \text{in}\quad H^1(\mathbb{R});&\\
\end{array}\right.\end{equation}
hence $\langle Q_hv_v,\tau_h\varphi\rangle_{L^2(\mathbb{R})}\rightarrow \langle v,\varphi\rangle_{L^2(\mathbb{R})}.$ On the other hand, we have
$$\sum_iv_i\cdot\int_{x_i}^{x_{i+1}}\varphi''(x+\rho)dx \leq |v_h|_h\|\varphi''\|_{L^2(\mathbb{R})}.$$ 
It follows that  
\begin{eqnarray*}
\left| \frac{1}{h}\int_0^h\int_0^s(s-\frac{h}{2})\left(\sum_iv_h(x_i)\cdot\int_{x_i}^{x_{i+1}}\varphi''(x+\rho)dx\right)d\rho ds\right|&\leq&
h^2|v_h|_h\|\varphi''\|_{L^2(\mathbb{R})}\\
&\leq& Ch\|\varphi''\|_{L^2(\mathbb{R})},
\end{eqnarray*}
and thus the proof of 1. is completed.  The proof of 2 is similar to that of 1.
\end{proof}

We establish now a compactness result which will be useful in the proofs of main theorems.
\begin{lemma}\label{compa}
 Let $T>0$ and $\{u_h\}_h$ be a sequence of elements in $L^\infty(0,T,H^1_{loc}(\mathbb{R})).$ Assume that $\{u_h\}_h$ is 
 bounded in $L^\infty(0,T,H^1_{loc}(\mathbb{R}))$ such that the sequence $\{\partial_tu_h\}_h$ is bounded in $L^\infty(0,T,H^{-1}(\mathbb{R})).$
 Then we can extract from $\{u_h\}_h$ a  subsequence converging in $\mathcal{C}(0,T,L^2_{loc}(\mathbb{R})).$ 
\end{lemma}
\begin{proof}
The proof is a consequence of the following proposition 
\begin{proposition}(\cite{compacite})
 Let $X,B$ and $Y$ be three Banach spaces such that $X\subset B\subset Y.$ Assume that the embedding $X\subset B$ is compact. 
 Let $F$ be some bounded subset in $L^\infty(0,T,X)$ such that the subset $G = \{\partial_tf, \quad f\in F\}$ is bounded in $L^r(0,T,Y),$
 with $1<r\leq\infty.$ Then $F$ is relatively compact in $\mathcal{C}(0,T,B).$  
\end{proposition}
We denote by $I_k=]-k,k[$ with $k\in \mathbb{N}.$ We consider the three spaces $X=H^1(I_k),$ $B=L^2(I_k)$ and
$Y = H^{-1}(I_k).$  The embedding $H^1(I_k) \subset L^2(I_k)$ is compact, hence using previous proposition, we can extract from 
$\{u_h\}_h$ a subsequence (depending on $k$) which converges in $\mathcal{C}(0,T,L^2(I_k)).$ Thus the diagonal subsequence of Cantor converges 
in $C(0,T,L^2(I_k))$ for all $k \in \mathbb{N}$.
\end{proof}


\subsection{Proof of Theorem \ref{th_ex_np}}
We construct a weak solution for the system  
\begin{equation}\label{LIAg}
\left\{\begin{array}{lr}
\partial_tu=\partial_x\left(u\wedge g(x)\partial_x u\right)=
u\wedge\partial_x\left(g\partial_x u\right),& t\geq 0,\quad x\in \mathbb{R},\\
u(0,x)=u_0(x),&
\end{array}\right.
\end{equation}
as a limit, when $h\rightarrow 0$, of a sequence  $\{u_h\}_h$ of solutions to the semi-discrete system 
\begin{equation}\label{LIAgD}
\left\{\begin{array}{lr}
\frac{du_h}{dt}=D^+\left(u_h\wedge g_hD^-u_h\right)=
u_h\wedge D^+\left(g_hD^-u_h\right),& t\geq 0,\\
u_h(0)=u_h^0,&
\end{array}\right.
\end{equation}
where $u_h^0=\{u_h^0(x_i)\}_i\in (\mathbb{R}^3)^{\mathbb{Z}_h}$ with $|u_h^0(x_i)|=1$ and $g_h=\{ g(t,x_i)\}_i.$

\begin{proposition}\label{ex_sol_Dg}
Let $u_h^0=\{u_h^0(x_i)\}_i\in (\mathbb{R}^3)^{\mathbb{Z}_h}$ be such that $|u_h^0(x_i)|=1,$ and $D^+u_h^0\in L_h^2.$ Let \\ 
$g\in W^{1,\infty}(\R^+,L^\infty(\mathbb{R}))$ be such that $g\geq\alpha$ for some $\alpha>0$. Then equation  (\ref{LIAgD}) has a global solution  
$u_h=\{u_h(x_i)\}_i\in \mathcal{C}^1(\mathbb{R}^+,(\mathbb{R}^3)^{\mathbb{Z}_h})$ with $|u_h(t,x_i)|=1$ and $D^+u_h\in \mathcal{C}^1(\mathbb{R}^+,L_h^2).$
\end{proposition}
\begin{proof}
Let $h>0.$ We provide the space  
$$E_h=\{v_h \in (\mathbb{R}^3)^{\mathbb{Z}_h},\quad v_h\in L_h^\infty\quad \text{and}\quad D^+v_h\in L_h^2\},$$ 
with the norm 
$$\|v_h\|_h=|v_h|_{L_h^\infty}+|D^+v_h|_h,\quad \forall v_h\in E_h,$$
for which the space $(E_h,\|.\|_h)$ is a Banach space. Let $R>0$ and $\Omega = B_{E_h}(u^0_h,R).$
We define the function
\begin{equation*}\left\{\begin{array}{lr}
F:\Omega \rightarrow E_h:\quad v_h\mapsto F(v_h),&\\
(F(v_h))(x_i)=D^+(v_h\wedge(g_hD^-v_h))(x_i)=\frac{1}{h^2}\left(g(x_i)v_h(x_i)\wedge v_h(x_{i-1})-g_h({i+1})v_h(x_{i+1})\wedge v_h(x_i)\right).&
\end{array}\right.\end{equation*}
In what follows we denote $\beta = \|g\|_{L^\infty(\mathbb{R})}.$ Let $u_h, v_h\in \Omega.$ We have on the one hand 
\begin{eqnarray*}
F(u_h)(x_i)-F(v_h)(x_i)&=&\frac{g_h(x_i)}{h^2}\left[u_h(x_i)\wedge(u_h(x_{i-1})-v_h(x_{i-1}))+(u_h(x_i)-v_h(x_i))\wedge v_h(x_i)\right]\\
&&+\frac{g_h(x_{i+1})}{h^2}\left[(v_h(x_{i+1})-u_h(x_{i+1}))\wedge v_h(x_i)+u_h(x_{i+1})(v_h(x_i)-u_h(x_i))\right],
\end{eqnarray*}
then 
\begin{equation}\label{anx_nrm_1}
|F(v_h)-F(u_h)|_{L_h^\infty}\leq\frac{4\beta}{h^2}(R+\|u^0_h\|_h)|v_h-u_h|_{L_h^\infty}.
\end{equation}
On the other hand, using Lemma \ref{Pd} we get 
\begin{eqnarray*}
|D^+(F(v_h)-F(u_h))|_h&=&|D^+[D^+(g_h(v_h\wedge D^-v_h-u_h\wedge D^-u_h))]|_h\\
&\leq& \frac{4\beta}{h^2}|v_h\wedge D^-v_h-u_h\wedge D^-u_h|_h\\
&\leq& \frac{4\beta}{h^2}(|v_h|_{L_h^\infty}|D^-(v_h-u_h)|_h+|D^-u_h|_h|D^-(v_h-u_h)|_h)\\
&\leq& \frac{4\beta}{h^2}(R+\|u^0_h\|_h)(|D^-(v_h-u_h)|_h+|D^-(v_h-u_h)|_h).\\
\end{eqnarray*}
It follows that 
\begin{equation}\label{anx_nrm_2}
|F(v_h)-F(u_h)|_h\leq\frac{4\beta}{h^2}(R+\|u^0_h\|)\|v-u\|_h.
\end{equation}
Then, combining (\ref{anx_nrm_1}) and (\ref{anx_nrm_2}), we deduce  that
$$\|F(v_h)-F(u_h)\|_h\leq\frac{8\beta}{h^2}(R+\|u^0_h\|_h)\|v_h-u_h\|_h.$$
Thus $F$ is locally Lipschitz-continuous and  Cauchy-Lipschitz theorem holds. Hence there 
exists $T^\ast\in\mathbb{R}_\ast^+\cup \{+\infty\}$ and $u_h:[0,T^\ast[\rightarrow(E_h,\|\|_h)$  satisfying (\ref{LIAgD}). Taking
the usual $\mathbb{R}^3-$scalar product in (\ref{LIAgD}) with $u_h,$ we find that $\frac{d}{dt}|u_h(t,x_i)|=0,$ hence $|u_h(t,x_i)|=|u^0_h(x_i)|=1$
on $[0,T^\ast[.$ Then we have $\|u_h\|_h=1+|D^+u_h|_h$ which gives $T^\ast$ the following characterisation 
$$\limsup_{t\rightarrow T^\ast}|D^+u_h(t)|_h =+\infty \quad \text{if}\quad T^\ast < +\infty.$$
Taking the $L_h^2-$scalar product in (\ref{LIAgD}) with $D^+\left(g_hD^-u_h\right),$ we get 
$$\frac{d}{dt}\sum_ig_h|D^-u_h(x_i)|^2(t,x_i)= \sum_i\partial_tg(t,x_i)|D^-u_h(x_i)|^2(t,x_i),$$ 
from which and by using the Gr\"{o}nwall lemma, we obtain
$$|D^+u_h(t)|_h=|D^-u_h(t)|_h\leq \sqrt{\frac{\beta}{\alpha}}|D^+u_h^0|_h\exp\left(\frac{\beta_1t}{2\alpha}\right)\quad \forall t\in [0,T^\ast[.$$
This means that $\lim_{t\rightarrow T^\ast}\|u_h\|_h\neq +\infty,$ hence we finally get $T^\ast=+\infty.$
\end{proof}

In what follows, we consider $T>0$ fixed.  For each sequence $\{v_h\}_h$ of elements in $L_h^2$, we have 
$\left(\frac{du_h}{dt},v_h\right)_h=-\left(u_h\wedge g_hD^-u_h,D^-v_h\right)_h,$ hence 
\begin{equation}\label{dert_borne}
\left|\frac{du_h}{dt}\right|_{H_h^{-1}}\leq \beta\sqrt{\frac{\beta}{\alpha}}|D^+u_h^0|\exp\left(\frac{\beta_1t}{2\alpha}\right).
\end{equation}
Let $\{u_h^0\}_h$ be a sequence satisfying
\begin{equation}\label{cond_di}\left\{\begin{array}{lr}
Q_hu_h^0\rightarrow u_0\quad\text{in}\quad L_{loc}^2(\mathbb{R}),&\\
Q_hD^+u_h^0 \rightarrow \frac{du_0}{dx}\quad\text{in}\quad L^2(\mathbb{R}).&\\
\end{array}\right.\end{equation}
Then we have
\begin{lemma}\label{lem_born1}
The sequence of solutions $\{u_h\}_h$ satisfying (\ref{LIAgD}), with initial data $\{u_h^0\}_h$ satisfying (\ref{cond_di}), has the properties\\
i) $\{\partial_tP_hu_h\}_h$ is bounded in $L^\infty(0,T,H^{-1}(\mathbb{R})).$\\
ii) $\{P_hu_h\}_h$ is bounded in $L^\infty(0,T,H_{loc}^1(\mathbb{R})).$
\end{lemma}
\begin{proof}
Property i) is an immediate result of  (\ref{dert_borne}) and Lemma \ref{lem_born}.\\
ii) Let $I=[a,b]\subset\mathbb{R}.$ Then we have 
\begin{eqnarray*}
\|P_hu_h\|^2_{H^1(I)}&=&\sum_i\int_{x_i}^{x_{i+1}}\left|\frac{x_i-x}{h}u_h(x_i)+\frac{x-x_i}{h}u_h(x_{i+1})\right|^2dx+
\sum_ih\left|\frac{u_h(x_i)-u_h(x_{i+1})}{h}\right|^2dx\\
&\leq&\sum_i\frac{h}{3}\left(|u_h(x_i)|^2+|u_h(x_{i+1})|^2+u_h(x_i)u_h(x_{i+1})\right)+|D^+u_h|^2_h\\
&\leq& b-a+2h+|D^+u_h^0|^2_h,
\end{eqnarray*}
where the sequence $\{|D^+u_h^0|_h\}_h$ is bounded, since $Q_hD^+u_h^0 \rightarrow \frac{du_0}{dx}\quad\text{in}\quad L^2(\mathbb{R}).$
\end{proof}

 Since $\{P_hu_h\}_h$ and $\{\partial_tP_hu_h\}_h$ are bounded in  
$L^\infty(0,T,H^1_{loc}(\mathbb{R}))$ and $L^\infty(0,T,H^{-1}(\mathbb{R}))$ respectively and in view of Lemma \ref{compa}, there exists a subsequence 
 $\{u_h\}_h$ and $u$ such that $\{P_hu_h\}_h$ converges to $u$ in $L^2(0,T,L^2_{loc}(\mathbb{R}))$ and almost everywhere. Moreover,
$\{\partial_tP_hu_h\}_h$ converges to $\partial_tu$ in  $L^\infty(0,T,H^{-1}(\mathbb{R}))$ weak star. 
The sequence $\{Q_hu_h\}$ converges also to $u$ almost everywhere. To show that the second member
$\{P_hD^+\left(u_h\wedge g_hD^-u_h\right)\}_h$ converges to $\partial_x\left(u\wedge g(x)\partial_xu\right)$, we note first that by Lemma
\ref{lem_cv_l2f}, the two sequences  $\{P_h(u_h\wedge g_hD^-u_h)\}_h$ and $\{Q_h(u_h\wedge g_hD^-u_h)\}_h$ converge  to the same limit in $L^\infty(0,T,L^2(\mathbb{R}))$
weak star. Since $$Q_h(u_h\wedge g_hD^-u_h)=Q_hu_h\wedge (Q_hg_hQ_hD^-u_h),$$ and
\begin{equation}\left\{\begin{array}{lr}
Q_hg_h\rightarrow g \quad \text{almost everywhere},&\\
Q_hu_h\rightarrow u\quad \text{almost everywhere},&\\
Q_hD^-u_h \rightarrow \partial_x u\quad\text{in}\quad L^\infty(0,T,L^2(\mathbb{R}))\quad \text{weak star},&\\
\end{array}\right.\end{equation}
we have 
$$Q_h(u_h\wedge g_hD^-u_h)\rightarrow u\wedge(g\partial_x u) \quad\text{in}\quad L^\infty(0,T,L^2(\mathbb{R}))\quad \text{weak star},$$
and
\begin{equation}\left\{\begin{array}{lr}
P_h(u_h\wedge g_hD^-u_h)\rightarrow u\wedge(g\partial_xu) \quad\text{in}\quad L^\infty(0,T,L^2(\mathbb{R}))\quad \text{weak star},&\\
\partial_xP_h(u_h\wedge g_hD^-u_h)\rightarrow \partial_x\left(u\wedge(g\partial_xu)\right) \quad\text{in}\quad L^\infty(0,T,H^{-1}(\mathbb{R}))\quad \text{weak star}.&\\
\end{array}\right.\end{equation}
It is clear that $$Q_hD^+\left(u_h\wedge g_hD^-u_h\right)=\partial_xP_h(u_h\wedge g_hD^-u_h),$$ 
then using lemma \ref{lem_cv_h-1}, the sequence $\{P_hD^+\left(u_h\wedge g_hD^-u_h\right)\}_h$ converges to
$\partial_x\left(u\wedge(g\partial_xu)\right)$ in $L^\infty(0,T,H^{-1}(\mathbb{R}))$ weak star. 

When $g=g(x)$ does not depend on time, we have $$\frac{d}{dt}\int_{\R}g(x)|\partial_xu(t,x)|^2dx =0,$$
then $$\|\partial_xu(t)\|^2_{L^2(\R)}\leq \frac{\|g\|_{L^\infty(\R)}}{\alpha}\left\|\frac{du_0}{dx}\right\|_{L^2(\R)},$$
and $u\in L^\infty(\R^+, H^1_{loc}(\mathbb{R})).$ Thus the proof of Theorem \ref{th_ex_np} is completed.

\subsection{Proof of Theorem \ref{th_ex_p}}
In this proof we use, without details, the same techniques of previous proof.
Let $l>0$ and $T>0$. We construct a solution $u\in L^\infty(0,T,H^1(\mathbb{T}^l,S^2))$ for the system
\begin{equation}\label{LIAgp}
\left\{\begin{array}{lr}
\partial_tu=\partial_x\left(u\wedge g\partial_xu\right)=
u\wedge \partial_x\left(g\partial_xu\right),& t\geq 0,\quad x\in \mathbb{T}^l,\\
u(0,x)=u_0(x).&
\end{array}\right.
\end{equation}
as a limit, when $h\rightarrow 0$, of a sequence $\{u_h =\{u_h(x_i)\}_i\in P_{l,N}\}_h$ (with $h=\frac{l}{N}$) 
of solutions to the semi-discrete system
\begin{equation}\label{LIAgDp}
\left\{\begin{array}{lr}
\frac{du_h}{dt}=D^+\left(u_h\wedge g_hD^-u_h\right)=
u_h\wedge D^+\left(g_hD^-u_h\right),& t> 0,\\
u_h(0)=u_h^0,&\\
u_h(t,x_0)=u_h(t,x_N),\quad t\geq0,&
\end{array}\right.
\end{equation}
with $|u_h(x_i)^0|=1,$ and $g_h=\{g(t,x_i)\}_i$ such that $g(t,x_0)=g(t,x_N).$

\begin{proposition}\label{ex_sol_Dgp}
Let $u_h^0\in P_{l,N}$ (with $h=\frac{l}{N}$) be such that $|u_h^0(x_i)|=1,$ and $g\in W^{1,\infty}(\R^+,L^\infty(\mathbb{T}^l))$ be 
such that  $g\geq\alpha$ for some $\alpha >0$. Then there exists a solution $u_h=\{u_h(x_i)\}_i\in \mathcal{C}^1(\mathbb{R}^+, P_{l,N})$  
for (\ref{LIAgDp}) with $|u_h(t,x_i)|=1$ for every $i.$
\end{proposition}
\begin{proof}
Let $l>0$ and $N\in \mathbb{N}.$ We denote $h=\frac{l}{N}.$ We provide the space $P_{l,N}$ with the norm
$$|v_h|_{L_h^\infty}=\sup_{i\in\mathbb{Z}}|v_h(x_i)|,\quad \forall v_h\in P_{l,N},$$
which makes $(P_{l,N},|.|_{L_h^\infty})$ a Banach space. Let $R>0$ and $\Omega = B_{P_{l,N}}(u^0_h,R).$
We define the function $F:\Omega \rightarrow P_{l,N}$ by 
\begin{eqnarray*}
(F(v_h))(x_i)&=&D^+(v_h\wedge(g_hD^-v_h))(x_i)\\
&=& \frac{1}{h^2}\left(g_h(x_i)v_h(x_i)\wedge v_h(x_{i-1})-g_h(x_{i+1})v_h(x_{i+1})\wedge v_h(x_i)\right).
\end{eqnarray*}
Then we follow the same steps followed to demonstrate Proposition \ref{ex_sol_Dg}.
\end{proof}

The rest of proof is similar to that of Theorem \ref{th_ex_np} and requires  
property (\ref{IPPdp}) and results of Lemmas \ref{lem_born}, \ref{lem_cv_l2f} and \ref{lem_cv_h-1}.

\subsection{Proof of Theorem \ref{th_ex_npf}}
We denote
$$\Delta_{g_h}v_h = D^+(g_hD^-v_h)=D^-(\tau^+g_hD^+v_h),\quad D^2= D^+D^-=D^+D^-, D^3=D^+D^-D^+,$$
and $g_h^t= \{\partial_tg(t,x_i)\}_i$. Since $g$ is given in $ W^{1,\infty}(\R^+,W^{3,\infty}(\mathbb{R},\mathbb{R})),$ 
then there exist $\beta,\beta_1, \beta',\beta'_1, \beta'',\beta''_2$ and $\beta'''$ such that
$$\left\{\begin{array}{lr}
|g_h|_{L^\infty_h}\leq \beta,\quad |g^t_h|_{L^\infty_h}\leq \beta_1&\\
|D^+g_h|_{L^\infty_h} =|D^-g_h|_{L^\infty_h}\leq \beta',\quad |D^+g^t_h|_{L^\infty_h} =|D^-g^t_h|_{L^\infty_h}\leq \beta'_1&\\
|D^2g_h|_{L^\infty_h}\leq \beta'',\quad |D^2g^t_h|_{L^\infty_h}\leq \beta''_2&\\
|D^3g_h|_{L^\infty_h}\leq \beta'''.&
\end{array}\right.
$$
Our proof include  several steps
\subsubsection{Step 1}
In this step, we establish two a priori estimates in $\frac{du_h}{dt},$ $D^-\frac{du_h}{dt},$ $\Delta_{g_h}u_h$ and $D^-\Delta_{g_h}u_h.$\\
We start by proving that
\begin{equation}
\frac{d}{dt}\left(\left|\frac{du_h}{dt}\right|_h^2+|\Delta_{g_h}u_h|_h^2 \right)\leq C_1\left(\left|\frac{du_h}{dt}\right|_h^2+|\Delta_{g_h}u_h|_h^2 \right)^2 +C_2,
\end{equation}
where $C_1$ and $C_2$ are two positive constants independent of $h.$
For any two sequences $u_h=\{u_h(x_i)\}_i$ and $v_h =\{u_h(x_i)\}_i,$ we have
\begin{eqnarray}\label{delta_gh}
\Delta_{g_h}(u_hv_h)&=& D^+(g_h\tau^-v_hD^-uh+g_hu_hD^-v_h)\notag\\
&=& \tau^+\tau^-v_h\Delta_{g_h}u_h+g_hD^+(\tau^-v_h)D^-u_h+\tau^+(g_hD^-v_h)D^+u_h+u_h\Delta_{g_h}v_h\notag\\
&=&v_h\Delta_{g_h}uh+g_hD^-v_hD^-u_h+\tau^+g_hD^+v_hD^+u_h+u_h\Delta_{g_h}v_h.
\end{eqnarray}
We derive (\ref{LIAgD}) with respect to $t$
\begin{equation}\label{derivu_htt}
\frac{d^2u_h}{dt^2}=(u_h\wedge\Delta_{g_h}u_h)\wedge\Delta_{g_h}u_h+u_h\wedge\Delta_{g_h}(u_h\wedge\Delta_{g_h}u_h)+u_h\wedge\Delta_{g^t_h}u_h.
\end{equation}
Using (\ref{delta_gh}) and $|u_h(t,x_i)|=1$ together with  equation (\ref{derivu_htt}), we get 
\begin{eqnarray}\label{derivu_htt1}
\frac{d^2u_h}{dt^2}&=& (u_h\cdot\Delta_{g_h}u_h )\Delta_{g_h}u_h-|\Delta_{g_h}u_h|^2u_h\notag \\
&&+u_h\wedge(g_hD^-u_h\wedge D^-\Delta_{g_h}u_h+\tau^+g_hD^+u_h\wedge D^+\Delta_{g_h}u_h+u_h\wedge \Delta_{g_h}^2u_h)\notag\\
&=& u_h\wedge\Delta_{g^t_h}u_h+(u_h\cdot\Delta_{g_h}u_h )\Delta_{g_h}u_h-|\Delta_{g_h}u_h|^2u_h +(u_h\cdot\Delta_{g_h}^2u_h )u_h-\Delta_{g_h}^2u_h\notag\\
&& + E,
\end{eqnarray}
where 
\begin{eqnarray*}
E&=& g_hu_h\wedge(D^-u_h\wedge D^-\Delta_{g_h}u_h)+\tau^+g_hu_h\wedge(D^+u_h\wedge D^+\Delta_{g_h}u_h)\\
&=& g_h(u_h\cdot D^-\Delta_{g_h}u_h)D^-u_h+\tau^+g_h(u_h\cdot D^+\Delta_{g_h}u_h)D^+u_h\\
&&-g_h(u_h\cdot D^-u_h)D^-\Delta_{g_h}u_h-\tau^+g_h(u_h\cdot D^+u_h)D^+\Delta_{g_h}u_h.
\end{eqnarray*}
Furthermore, we have 
$$u_h\cdot D^\pm u_h =\mp \frac{h}{2}(D^\pm u_h)^2,$$ 
hence
\begin{eqnarray*}
\tau^+g_h(u_h\cdot D^+u_h)D^+\Delta_{g_h}u_h&=&-\frac{h}{2}\tau^+g_h(D^+u_h)^2D^+\Delta_{g_h}u_h\\
&=& -\frac{h}{2}\left\{D^-[(D^+u_h)^2\tau^+(g_h\Delta_{g_h}u_h)]-D^-(\tau^+g_h(D^+u_h)^2)\Delta_{g_h}u_h\right\}\\
&=& -\frac{h}{2}\left\{D^+[g_h(D^-u_h)^2\Delta_{g_h}u_h]-D^+(g_h(D^-u_h)^2)\Delta_{g_h}u_h\right\},
\end{eqnarray*}
and
\begin{eqnarray*}
g_h(u_h\cdot D^-u_h)D^-\Delta_{g_h}u_h&=&\frac{h}{2}g_h(D^-u_h)^2D^-\Delta_{g_h}u_h\\
&=& \frac{h}{2}\left\{D^+[g_h(D^-u_h)^2\tau^-\Delta_{g_h}u_h]-D^+(g_h(D^-u_h)^2)\Delta_{g_h}u_h\right\},
\end{eqnarray*}
which together give  
\begin{equation}\label{derivu_h1}
-\tau^+g_h(u_h\cdot D^+u_h)D^+\Delta_{g_h}u_h-g_h(u_h\cdot D^-u_h)D^-\Delta_{g_h}u_h=\frac{h^2}{2}D^+[g_h(D^-u_h)^2D^-\Delta_{g_h}u_h].
\end{equation}
On the other hand, we have 
\begin{eqnarray*}
u_h\cdot\Delta_{g_h}u_h&=&u_h\cdot(D^+g_hD^-u_h+\tau^+g_hD^+D^-u_h)\\
&=&\frac{h}{2}D^+g_h(D^-u_h)^2-\frac{1}{2}\tau^+g_h((D^-u_h)^2+(D^+u_h)^2)\\
&=&-\frac{1}{2}(g_h(D^-u_h)^2+\tau^+g_h(D^+u_h)^2),
\end{eqnarray*}
hence 
\begin{eqnarray}\label{derivu_h2}
u_h\cdot D^\pm\Delta_{g_h}u_h&=& D^\pm(u_h\cdot\Delta_{g_h}u_h)-D^\pm u_h\cdot\tau^\pm(\Delta_{g_h}u_h)\notag\\
&=&-\frac{1}{2}D^\pm\left(g_h(D^-u_h)^2+\tau^+g_h(D^+u_h)^2\right)-D^\pm u_h\cdot\tau^\pm(\Delta_{g_h}u_h).
\end{eqnarray}
Combining (\ref{derivu_h1}) and (\ref{derivu_h2}) we find that 
\begin{eqnarray*}
E&=& \frac{h^2}{2}D^+[g_h(D^-u_h)^2D^-\Delta_{g_h}u_h]\\
&& -\frac{1}{2}g_hD^-\left(g_h(D^-u_h)^2+\tau^+g_h(D^+u_h)^2\right)D^-u_h-g_h(D^- u_h\cdot\tau^-\Delta_{g_h}u_h)D^-u_h\\
&&-\frac{1}{2}\tau^+g_hD^+\left(g_h(D^-u_h)^2+\tau^+g_h(D^+u_h)^2\right)D^+u_h-\tau^+g_h(D^+u_h\cdot\tau^+\Delta_{g_h}u_h)D^+u_h.
\end{eqnarray*}
Taking  the $L_h^2-$scalar product in (\ref{derivu_htt1}) with $\frac{du_h}{dt}$ and using  $u_h\cdot\frac{du_h}{dt}=0,$
$\Delta_{g_h}u_h\cdot\frac{du_h}{dt}=0$ and 
$$\Delta_{g_h}(\frac{du_h}{dt})= \frac{d}{dt}\Delta_{g_h}(u_h)-\Delta_{g^t_h}u_h,$$
we obtain by integration by parts
$$\frac{1}{2}\frac{d}{dt}\left(\left|\frac{du_h}{dt}\right|_h^2+|\Delta_{g_h}u_h|_h^2 \right)=J_1+J_2+I_1+I_2^+ +I_2^-+I_3^++I_3^-,$$
where 
$$J_1 = (\Delta_{g^t_h}u_h,\Delta_{g_h}u_h)_h,$$
$$J_2 = (u_h\wedge\Delta_{g^t_h}u_h,u_h\wedge\Delta_{g_h}u_h)_h,$$
$$I_1=\frac{h^2}{2}\left(D^+[g_h(D^-u_h)^2D^-\Delta_{g_h}u_h],\frac{du_h}{dt}\right)_h,$$
$$I_2^+=-\frac{1}{2}\left(\tau^+g_hD^+\left(g_h(D^-u_h)^2+\tau^+g_h(D^+u_h)^2\right)D^+u_h,\frac{du_h}{dt}\right)_h,$$
$$I_2^-=-\frac{1}{2}\left(g_hD^-\left(g_h(D^-u_h)^2+\tau^+g_h(D^+u_h)^2\right)D^-u_h,\frac{du_h}{dt}\right)_h,$$
$$I_3^+=-\frac{1}{2}\left(\tau^+g_h(D^+u_h\cdot\tau^+\Delta_{g_h}u_h)D^+u_h,\frac{du_h}{dt}\right)_h,$$
$$I_3^-=-\frac{1}{2}\left(g_h(D^- u_h\cdot\tau^-\Delta_{g_h}u_h)D^-u_h,\frac{du_h}{dt}\right)_h.$$
To estimate these terms, we use essentially  the H\"{o}lder inequality and Lemmas \ref{Pd} and \ref{Sobolev_discret}. We start by 
\begin{eqnarray}\label{major0}
 J_1+J_2 &\leq& 2|\Delta_{g^t_h}u_h|_h|\Delta_{g_h}u_h|_h\notag\\
 &\leq& 2(\beta'_1|D^+u_h|_h+\beta_1|D^2u_h|_h) |\Delta_{g_h}u_h|_h.
\end{eqnarray}
Then, we have on the one hand 
\begin{eqnarray}\label{major1}
I_1 &\leq& \frac{h^2}{2}|D^+g_h(D^-u_h)^2D^-\Delta_{g_h}u_h|_h\left|\frac{du_h}{dt}\right|_h \notag\\
&\leq& h|g_h(D^-u_h)^2D^-\Delta_{g_h}u_h|_h\left|\frac{du_h}{dt}\right|_h\notag\\
&\leq& h\beta|D^-u_h|_{L_h^\infty}^2|D^-\Delta_{g_h}u_h|_h\left|\frac{du_h}{dt}\right|_h\notag \\
&\leq& 2C\beta|D^-u_h|_{H_h^1}^2|\Delta_{g_h}u_h|_h\left|\frac{du_h}{dt}\right|_h,
\end{eqnarray}
and on the other hand $I_2^+= I_{21}^++I_{22}^+,$ with
$$I_{21}^+=-\frac{1}{2}\left(\tau^+g_hD^+(g_h(D^-u_h)^2)D^+u_h,\frac{du_h}{dt}\right)_h,\quad
I_{22}^+=-\frac{1}{2}\left(\tau^+g_hD^+(\tau^+g_h(D^+u_h)^2)D^+u_h,\frac{du_h}{dt}\right)_h.$$
Moreover, 
$$I_{21}^+=-\frac{1}{2}\left(\tau^+g_h\left(D^+g_h(D^-u_h)^2+\tau^+g_h(D^-+\tau^+D^-)u_h\cdot D^+D^-u_h\right)D^+u_h,\frac{du_h}{dt}\right)_h,$$
hence
\begin{eqnarray*}
I_{21}^+&\leq& \frac{1}{2}\beta\left(\beta'|D^-u_h|_h+2\beta|D^+D^-u_h|_h\right)|D^-u_h|_{L_h^\infty}^2\left|\frac{du_h}{dt}\right|_h\\
&\leq& \frac{1}{2}C\beta\left(\beta'|D^-u_h|_h+2\frac{\beta}{\alpha}|\Delta_{g_h}u_h|_h\right)|D^-u_h|_{H_h^1}^2\left|\frac{du_h}{dt}\right|_h.
\end{eqnarray*}
Similarly, we find that  
$$I_{22}^+\leq \frac{1}{2}C\beta\left(\beta'|D^-uh|_h+2\frac{\beta}{\alpha}|\Delta_{g_h}u_h|_h\right)|D^-u_h|_{H_h^1}^2\left|\frac{du_h}{dt}\right|_h,$$
then
\begin{equation}\label{major2}
I_2^+\leq C\beta\left(\beta'|D^-u_h|_h+2\frac{\beta}{\alpha}|\Delta_{g_h}u_h|_h\right)|D^-u_h|_{H_h^1}^2\left|\frac{du_h}{dt}\right|_h.
\end{equation}
For $I_3^+$ we easily note that
\begin{equation}\label{major3}
I_3^+\leq \frac{1}{2}C\beta|D^-u_h|_{H_h^1}^2|\Delta_{g_h}u_h|_h\left|\frac{du_h}{dt}\right|_h.
\end{equation}
The two terms $I_3^-$ and $I_2^-$ can be estimated in the same way followed to estimate $I_3^+$ and $I_2^+$. Since
\begin{eqnarray}\label{major4}
|D^-u_h|_{H_h^1}^2 &=& |D^-u_h|_h^2+|D^+D^-u_h|_h^2\notag\\
&\leq& |D^-u_h|_h^2+\frac{1}{\alpha^2}|\Delta_{g_h}u_h|_h^2,
\end{eqnarray}
we get by combining (\ref{major0}), (\ref{major1}), (\ref{major2}), (\ref{major3}), (\ref{major4}) and (\ref{dert_borne}) 
\begin{equation}\label{gronwall1}
\frac{d}{dt}\left(\left|\frac{du_h}{dt}\right|_h^2+|\Delta_{g_h}u_h|_h^2 \right)\leq 
C_1\left(\left|\frac{du_h}{dt}\right|_h^2+|\Delta_{g_h}u_h|_h^2 \right)^2 +C_2,
\end{equation}
where $C_1, C_2>0$ are two constants depending on $\alpha, \beta, \beta_1, \beta', \beta'_1$  and $|D^+u_h^0|_h.$

Then we establish an a priori estimate in $D^-\frac{du_h}{dt}$ and $D^-\Delta_{g_h}u_h.$
We denote $$A_{g_h}u_h=\frac{1}{2}(g_h(D^-u_h)^2+\tau^+g(D^+u_h)^2).$$
We have found that 
\begin{equation}\label{deriv7}
\frac{d^2u_h}{dt^2}+\Delta_{g_h}^2u_h=(u_h\cdot\Delta_{g_h}u_h )\Delta_{g_h}u_h-|\Delta_{g_h}u_h|^2u_h +(u_h\cdot\Delta_{g_h}^2u_h )u_h
+  u_h\wedge\Delta_{g^t_h}u_h+ E,
\end{equation}
where 
\begin{eqnarray}\label{deriv8}
E&=& \frac{h^2}{2}D^+[g_h(D^-u_h)^2D^-\Delta_{g_h}u_h]\notag\\
&& -g_hD^-(A_{g_h}u_h)D^-u_h-g_h(D^- u_h\cdot\tau^-\Delta_{g_h}u_h)D^-u_h\notag\\
&&-\tau^+g_hD^+(A_{g_h}u_h)D^+u_h-\tau^+g_h(D^+u_h\cdot\tau^+\Delta_{g_h}u_h)D^+u_h.
\end{eqnarray}
Moreover, using (\ref{delta_gh}), we get 
\begin{eqnarray}\label{deriv9}
u_h\cdot\Delta_{g_h}^2(u_h)&=& \Delta_{g_h}(u_h\cdot\Delta_{g_h}u_h)-|\Delta_{g_h}u_h|^2
-g_hD^-\Delta_{g_h}u\cdot D^-u_h-\tau^+g_hD^+\Delta_{g_h}u\cdot D^+u_h\notag\\
&=& -\Delta_{g_h}(A_{g_h}u_h)-|\Delta_{g_h}u_h|^2-g_hD^-\Delta_{g_h}u\cdot D^-u_h\notag\\
&& -\tau^+g_hD^+\Delta_{g_h}u\cdot D^+u_h.
\end{eqnarray}
Thus Combining (\ref{deriv7}), (\ref{deriv8}) and (\ref{deriv9}), we get
\begin{eqnarray}\label{deriv10}
\frac{d^2u_h}{dt^2}+\Delta_{g_h}^2u_h &=&-\Delta_{g_h}(A_{g_h}u_h)u_h-A_{g_h}u_h\Delta_{g_h}u_h-\tau^+g_hD^+(A_{g_h}u_h)D^+u_h-g_hD^-(A_{g_h}u_h)D^-u_h\notag\\
&& -g_h(D^- u_h\cdot\tau^-\Delta_{g_h}u_h)D^-u -g_hD^-\Delta_{g_h}u\cdot D^-u_h -|\Delta_{g_h}u_h|^2\notag\\
&& -\tau^+g_h(D^+u_h\cdot\tau^+\Delta_{g_h}u_h)D^+u_h-\tau^+g_hD^+\Delta_{g_h}u\cdot D^+u_h -|\Delta_{g_h}u_h|^2\notag\\
&& +\frac{h^2}{2}D^+[g_h(D^-u_h)^2D^-\Delta_{g_h}u_h]+u_h\wedge\Delta_{g^t_h}u_h,
\end{eqnarray}
where 
$$-g_h(D^- u_h\cdot\tau^-\Delta_{g_h}u_h)D^-u -g_hD^-\Delta_{g_h}u\cdot D^-u_h -|\Delta_{g_h}u_h|^2=-D^-(\tau^+g_h(D^+u_h\cdot\Delta_{g_h}u_h)u_h),$$
$$-\tau^+g_h(D^+u_h\cdot\tau^+\Delta_{g_h}u_h)D^+u_h-\tau^+g_hD^+\Delta_{g_h}u\cdot D^+u_h -|\Delta_{g_h}u_h|^2= -D^+(g_h(D^-u_h\cdot\Delta_{g_h}u_h)u_h).$$
We have 
\begin{equation*}
\left\{
\begin{array}{lr}
D^+u_h\cdot\Delta_{g_h}u_h= D^+g_h|D^+u_h|^2+\frac{1}{2}g_hD^+(|D^-u_h|^2)+\frac{h}{2}g_h|D^+D^-u_h|^2,&\\
D^-u_h\cdot\Delta_{g_h}u_h= D^-g_h|D^-u_h|^2+\frac{1}{2}\tau^+g_hD^-(|D^+u_h|^2)+\frac{h}{2}\tau^+g_h|D^+D^-u_h|^2,&
\end{array}
\right.
\end{equation*}
and 
\begin{equation*}
\left\{
\begin{array}{lr}
g_hD^+(|D^-u_h|^2)u_h= D^+(g_h|D^-u_h|^2u_h)-\tau^+(|D^-u_h|^2)D^+(g_hu_h),&\\
\tau^+g_hD^-(|D^+u_h|^2)u_h= D^-(\tau^+g_h|D^+u_h|^2u_h)-\tau^-(|D^+u_h|^2)D^-(g_hu_h),&\\
\end{array}
\right.
\end{equation*}
then
\begin{eqnarray*}
D^-(\tau^+g_h(D^+u_h\cdot\Delta_{g_h}u_h)u_h)&=& \frac{1}{2}\Delta_{g_h}(g_h|D^-u_h|^2u_h)+\frac{1}{2}D^-(\tau^+g_h|D^-u_h|^2[D^+g_hu_h-\tau^+g_hD^+u_h])\\
&&+\frac{h}{2}D^-(\tau^+g_h g_h|D^+D^-u_h|^2),
\end{eqnarray*}
and 
\begin{eqnarray*}
D^+(g_h(D^-u_h\cdot\Delta_{g_h}u_h)u_h)&=& \frac{1}{2}\Delta_{g_h}(\tau^+g_h|D^+u_h|^2u_h)+\frac{1}{2}D^+(g_h|D^+u_h|^2[D^+g_hu_h-g_hD^-u_h])\\
&&+\frac{h}{2}D^+(\tau^+g_h g_h|D^+D^-u_h|^2).
\end{eqnarray*}
Thus equation (\ref{deriv10}) can be rewritten as 
\begin{eqnarray}\label{deriv11}
\frac{d^2u_h}{dt^2}+\Delta_{g_h}^2u_h &=&-2\Delta_{g_h}((A_{g_h}u_h)u_h)+u_h\wedge\Delta_{g^t_h}u_h\notag\\
&& +\frac{1}{2}D^+\left(g_h|D^+u_h|^2[2g_hD^-u_h-D^+g_hu_h-D^-g_h\tau^-u_h]\right)\notag\\
&& -\frac{h}{2}(D^++D^-)(\tau^+g_h g_h|D^+D^-u_h|^2)+\frac{h^2}{2}D^+\left(g_h(D^-u_h)^2D^-\Delta_{g_h}u_h\right).
\end{eqnarray}
Applying operator $D^-$ on (\ref{deriv11}) and taking the $L_h^2-$scalar product with $g_hD^-\frac{du_h}{dt}$, we get, after 
integration by parts,
$$\frac{h}{2}\frac{d}{dt}\sum_i g_h(x_i)\left(\left|D^-\frac{du_h}{dt}(x_i)\right|^2+|D^-\Delta_{g_h}u_h(x_i)|^2\right)=I_1+I_2+I_3+I_4+J_1+J_2+J_3,$$
with 
$$I_1= -2\left(D^-\Delta_{g_h}((A_{g_h}u_h)u_h),g_hD^-\frac{du_h}{dt}\right)_h,$$
$$I_2=\frac{1}{2}\left(D^-D^+(g_h|D^+u_h|^2[2g_hD^-u_h-D^+g_hu_h-D^-g_h\tau^-u_h]),g_hD^-\frac{du_h}{dt}\right)_h,$$
$$I_3=-\frac{1}{2}\left(hD^-(D^++D^-)(\tau^+g_h g_h|D^+D^-u_h|^2),g_hD^-\frac{du_h}{dt}\right)_h,$$
$$I_4 = \frac{1}{2}\left(h^2D^-D^+[g_h(D^-u_h)^2D^-\Delta_{g_h}u_h],g_hD^-\frac{du_h}{dt}\right)_h,$$
$$J_1 = \left(D^-(u_h\wedge\Delta_{g^t_h}u_h,g_hD^-\frac{du_h}{dt})\right)_h,$$
$$J_2 = (g_hD^-\Delta_{g_h}u_h, D^-\Delta_{g^t_h}u_h)_h,$$
$$J_3 = \frac{h}{2}\frac{d}{dt}\sum_i g^t_h(x_i)\left(\left|D^-\frac{du_h}{dt}(x_i)\right|^2+|D^-\Delta_{g_h}u_h(x_i)|^2\right).$$

We start by estimating $J_1, J_2$ and $J_3.$ We have 
\begin{eqnarray}\label{estimA}
|J_1|&\leq& \beta|D^-u_h|_{L^\infty_h}(\beta_1|D^2u_h|_h+\beta'_1|D^+u_h|_h)\left|D^-\frac{du_h}{dt}\right|_h\notag\\
&& + \beta(2\beta'_1|D^2u_h|_h+\beta''_1|D^+u_h|_h+\beta|D^3u_h|_h)\left|D^-\frac{du_h}{dt}\right|_h,
\end{eqnarray}
\begin{equation}\label{estimB}
|J_2| \leq \beta(2\beta'_1|D^2u_h|_h+\beta''_1|D^+u_h|_h+\beta|D^3u_h|_h)\left|D^-\Delta_{g_h}u_h\right|_h,
\end{equation}
\begin{equation}\label{estimC}
|J_3| \leq \frac{1}{2}\beta_1\left(\left|D^-\Delta_{g_h}u_h\right|^2_h+\left|D^-\frac{du_h}{dt}\right|^2_h\right).
\end{equation}
For the term $I_2$, we have
\begin{eqnarray}\label{estim1}
|I_2|&\leq& \frac{1}{2}\beta\{2|D^2(g_h^2|D^+u_h|^2D^-u_h)|_h+ |D^2(g_hD^+g_h|D^+u_h|^2u_h)|_h\notag\\
&& + |D^2(g_hD^-g_h|D^+u_h|^2\tau^-u_h)|_h\}\left|D^-\frac{du_h}{dt}\right|_h,
\end{eqnarray}
and 
\begin{eqnarray}\label{estim2}
|D^2(g_h^2|D^+u_h|^2D^-u_h)|_h &\leq& C\{((\beta'^2+\beta\beta'')|D^+u_h|_{L^\infty_h}^2+\beta^2|D^2u_h|_{L^\infty_h}^2)|D^+u_h|_h\notag\\
&&+\beta\beta'|D^+u_h|_{L^\infty_h}^2|D^2u_h|_h+\beta^2|D^+u_h|_{L^\infty_h}^2|D^3u_h|_h\}.
\end{eqnarray}
We also have 
\begin{eqnarray}\label{estim3}
|D^2(g_hD^+g_h|D^+u_h|^2u_h)|_h &\leq& C\{((\beta\beta'''+2\beta'\beta'')|D^+u_h|_{L_h^\infty}+(\beta'^2+\beta\beta'')|D^+u_h|_{L_h^\infty}^2)|D^+u_h|_h\notag\\
&&+(\beta\beta'|D^+u_h|_{L_h^\infty}^2+(\beta'^2+\beta\beta'')|D^+u_h|_{L_h^\infty})|D^2u_h|_h\notag\\
&&+\beta\beta'|D^+u_h|_{L_h^\infty}|D^3u_h|_h\}.
\end{eqnarray}
The term $|D^2(g_hD^-g_h|D^+u_h|^2\tau^-u_h)|_h$ can be bounded from above by the same term of the right-hand side of (\ref{estim3}). 
To find a suitable estimate for $I_1,$ we first write
\begin{eqnarray*}
D^-\Delta_{g_h}(g_h|D^-u_h|^2u_h)&=&D^2(g_hD^-(g_h|D^-u_h|^2u_h))\\
&=& D^2(g_h^2\tau^-|D^-u_h|^2D^-u_h+g_hD^-g_h\tau^-|D^-u_h|^2\tau^-u_h+g_h^2D^-(|D^-u_h|^2)u_h).
\end{eqnarray*}
Thus the two terms $|D^2(g_h^2\tau^-|D^-u_h|^2D^-u_h)|_h$ and $|D^2(g_hD^-g_h\tau^-|D^-u_h|^2\tau^-u_h)|_h$ can be bounded from above by the
members of right-hand sides of (\ref{estim2}) and (\ref{estim3}) respectively. For the term $D^2(g_h^2D^-(|D^-u_h|^2)u_h)$, we have 
\begin{equation}\label{estim33}
\left(D^2(g_h^2D^-(|D^-u_h|^2)u_h),g_hD^-\frac{du_h}{dt}\right)_h= I_{21}+\left(D^3(|D^-u_h|^2)u_h,g_h^3D^-\frac{du_h}{dt}\right)_h, 
\end{equation}
with
\begin{eqnarray}\label{estim4}
I_{21}&\leq& C\beta\{\beta^2|D^2u_h|_{L_h^\infty}^2|D^+u_h|h+((\beta\beta''+\beta'^2)|D^+u_h|_{L_h^\infty}+\beta\beta'|D^2u_h|_{L_h^\infty}
+\beta\beta'|D^+u_h|_{L_h^\infty}^2)|D^2u_h|_h\notag\\
&&+(\beta\beta'|D^+u_h|_{L_h^\infty}+\beta^2|D^2u_h|_{L_h^\infty})|D^3u_h|_h\}|D^-\frac{du_h}{dt}|_h.
\end{eqnarray}
Integrating by parts the second term of the right-hand side  of (\ref{estim33}), we obtain
\begin{eqnarray*}
\left(D^3(|D^-u_h|^2)u_h,g_h^3D^-\frac{du_h}{dt}\right)_h&=&-\left(D^2(|D^-u_h|^2)u_h,D^+g_h^3D^-\frac{du_h}{dt}\right)_h\\
&& -h\sum_ig^3(x_i)D^2(|D^-u_h|^2)(x_i)D^+(u_h\cdot D^-\frac{du_h}{dt})(x_i).
\end{eqnarray*}
Moreover, since $u_h\cdot\frac{du_h}{dt}=0,$ we have 
\begin{eqnarray*}
D^+(u_h\cdot D^-\frac{du_h}{dt})&=& D^+u_h\cdot D^+\frac{du_h}{dt}+u_h\cdot D^2\frac{du_h}{dt}-D^2(u_h\cdot\frac{du_h}{dt})\\
&=& -D^2u_h\cdot\frac{du_h}{dt}-D^-u_h\cdot D^-\frac{du_h}{dt}.
\end{eqnarray*}
Consequently, we get  
\begin{eqnarray}\label{estim5}
\left(D^3(|D^-u_h|^2)u_h,g_h^3D^-\frac{du_h}{dt}\right)_h&\leq& C\{ \beta^3|D^2u_h|_{L_h^\infty}^2|D^+u_h|_h
+\beta' \beta^2|D^2u_h|_{L_h^\infty}|D^2u_h|_h\notag\\
&& +(\beta'\beta^2|D^+u_h|_{L_h^\infty}+\beta^3|D^+u_h|_{L_h^\infty}^2)|D^3u_h|_h\}\left|D^-\frac{du_h}{dt}\right|_h\notag\\
&& +\beta^3\{|D^+u_h|_{L_h^\infty}|D^2u_h|_{L_h^\infty}|D^3u_h|_h\notag\\
&&+|D^2u_h|_{L_h^\infty}^2|D^2u_h|_h\}\left|\frac{du_h}{dt}\right|_h.
\end{eqnarray}
In view of the definition of $I_3$ and $I_4,$ we have  
$$|I_3|\leq h\beta|D^2(g_h\tau^+g_h|D^2u_h|^2)|_h\left|D^-\frac{du_h}{dt}\right|_h,$$
and 
$$|I_4|\leq\frac{1}{2}h^2\beta|D^2(g_h|D^-u_h|^2D^-\Delta_{g_h}u_h)|_h\left|D^-\frac{du_h}{dt}\right|_h,$$
where, applying  Lemma \ref{Pd}, we get 
$$h|D^2(g_h\tau^+g_h|D^2u_h|^2)|_h\leq 2|D^+(g_h\tau^+g_h|D^2u_h|^2)|_h,$$
and 
$$h^2|D^2(g_h|D^2u_h|^2D^-\Delta_{g_h}u_h)|_h\leq 4|g_h|D^-u_h|^2D^-\Delta_{g_h}u_h|_h,$$
which gives together with previous estimates of $I_3$ and $I_4$ 
\begin{equation}\label{estim6}
|I_3|\leq C\beta^2|D^2u_h|_{L_h^\infty}(\beta'|D^2u_h|_h+\beta|D^3u_h|_h)|D^-\frac{du_h}{dt}|_h,
\end{equation}
and $$|I_4|\leq 2\beta^2|D^-u_h|_{L_h^\infty}^2|D^-\Delta_{g_h}u_h|_h|D^-\frac{du_h}{dt}|_h.$$
Since 
\begin{equation}\label{D_Delta}
D^-\Delta_{g_h}u_h = D^2g_hD^-u_h+g_hD^3u_h+D^+g_hD^2u_h+D^-g_hD^-D^-u_h,
\end{equation}
we obtain 
\begin{equation}\label{estim7}
|I_4|\leq 2\beta^2|D^-u_h|_{L_h^\infty}^2(\beta''|D^-u_h|_h+2\beta'|D^2u_h|_h+\beta|D^3u_h|_h)\left|D^-\frac{du_h}{dt}\right|_h.
\end{equation}
Combining (\ref{estimA} - \ref{estim6}) and (\ref{estim7}), we finally get 
\begin{equation}\label{gronwall2}
\frac{1}{2}\frac{d}{dt}h\sum_i g_h(x_i)\left(\left|D^-\frac{du_h}{dt}(x_i)\right|^2+|D^-\Delta_{g_h}u_h(x_i)|^2\right)\leq CA_1A_2,
\end{equation}
with $A_1= |D^+u_h|_{L_h^\infty}+|D^+u_h|_{L_h^\infty}^2+|D^2u_h|_{L_h^\infty}+|D^2u_h|_{L_h^\infty}^2,$
$A_2 = |\frac{du_h}{dt}|_{H_h^1}^2+|D^2u_h|_{H_h^1}^2+|D^+u_h|_h^2$ and 
 $C>0$ is some constant depending on $\beta,\beta_1 \beta',\beta'_1, \beta'',\beta''_1$ and $\beta'''.$                                                          
\subsubsection{Step 2}
We construct the sequence $\{u_h^0\}_h$ such that 
\begin{equation}\label{cond_di1}\left\{\begin{array}{lr}
Q_hu_h^0\rightarrow u_0\quad\text{in}\quad L_{loc}^2(\mathbb{R}),&\\
Q_hD^+u_h^0 \rightarrow \frac{du_0}{dx}\quad\text{in}\quad L^2(\mathbb{R}),&\\
Q_hD^2u_h^0 \rightarrow \frac{d^2u_0}{dx^2}\quad\text{in}\quad L^2(\mathbb{R}),&\\
Q_hD^3u_h^0 \rightarrow \frac{d^3u_0}{dx^3}\quad\text{in}\quad L^2(\mathbb{R}),&\\
\end{array}\right.\end{equation}
then
\begin{lemma}\label{lem_born2}
There exists $T_1>0$ such that the sequences $\{\partial_tP_hu_h\}_h,$ $\{\partial_tP_hD^-u_h\}_h,$
$\{P_hD^2u_h\}_h$ and $\{P_hD^3u_h\}_h$ are bounded in $L^\infty(0,T_1,L^2(\mathbb{R})).$
\end{lemma}
\begin{proof}
Let $T> \frac{1}{\sqrt{C_1C_2}}.$ For  $t\in [0,T]$ we denote
$$G(t)=C_2T+\left|\frac{du_h}{dt}(0)\right|_h^2+|\Delta_{g_h}u_h(0)|_h^2
+C_1\int_0^t\left(\left|\frac{du_h}{dt}(\tau)\right|_h^2+|\Delta_{g_h}u_h(\tau)|_h^2 \right)^2d\tau,$$
where $C_1$ and $C_2$ are the constants of  inequality (\ref{gronwall1}), hence $\frac{1}{G}\in W^{1,\infty}(0,T)$ 
and in view of (\ref{gronwall1}) we have 
$$\left(\frac{1}{G(t)}\right)'\leq C_1,\quad \text{almost everywhere on}\quad ]0,T[.$$
then we have
$$C_1t+\frac{1}{G(t)}\geq \frac{1}{G(0)},\quad \forall t\in [0,T[,$$
and  
$$G(t)\leq \frac{G(0)}{1-C_1G(0)t},\quad \forall t\in [0,(C_1G(0))^{-1}[.$$
Since  
\begin{eqnarray*}
G(0)= C_2T+\left|\frac{du_h}{dt}(0)\right|_h^2+|\Delta_{g_h}u_h(0)|_h^2&\leq& 2|\Delta_{g_h}u_h(0)|_h^2+C_2T\\
&\leq& 4\beta'^2|D^+u_h^0|_h^2+4\beta^2|D^+D^-u_h^0|_h^2+C_2T,
\end{eqnarray*}
the sequences $\{|D^+u_h^0|_h\}_h$ and $\{|D^+D^-u_h^0|_h\}_h$ are bounded. Thus there exists $M>0$ such that 
$$4\beta'^2|D^+u_h^0|_h^2+4\beta^2|D^+D^-u_h^0|_h^2+C_2T\leq M,$$ then
$$G(0)^{-1}\geq M^{-1}>0.$$
Let $\tilde{T}=\frac{1}{2}(C_1M)^{-1}.$ Then, for all $t\in [0,\tilde{T}],$ we have 
\begin{equation}\label{gronwall3}
\left|\frac{du_h}{dt}\right|_h^2+|\Delta_{g_h}u_h|_h^2 \leq G(t)\leq \frac{M}{1-\frac{1}{2}M^{-1}G(0)}\leq 2M.
\end{equation}
In view of Corollary \ref{Sobolev_discret}, there exists $C>0$ such that 
$$|D^+u_h|_{L_h^\infty}\leq C |D^+u_h|_{H_h^1},\quad |D^2u_h|_{L_h^\infty}\leq C |D^2u_h|_{H_h^1}.$$
Thus combining (\ref{gronwall2}) and (\ref{gronwall3}), we have for all $t\in [0, \tilde{T}]$
\begin{equation}
\frac{1}{2}\frac{d}{dt}h\sum_i g_h(x_i)\left(\left|D^-\frac{du_h}{dt}(x_i)\right|^2+|D^-\Delta_{g_h}u_h(x_i)|^2\right)\leq
C_1(|D^-\frac{du_h}{dt}|_h^2+|D^-\Delta_{g_h}u_h|_h^2)^2+C_2,
\end{equation}
where $C_1, C_2>0$ depend on $\beta,$ $\beta_1,$ $\beta',$ $\beta'_1, $ $\beta'',$ $\beta''_1,$ $\beta''',$ $\alpha,$ and $M.$
Following the same argument in the previous part of this step, we find that there exists 
$K>0$ and $0<T_1\leq \tilde{T}$ such that, for all $t\in [0, T_1],$ we have 
\begin{equation}\label{gronwall4}
\left|D^-\frac{du_h}{dt}\right|_h+|D^-\Delta_{g_h}u_h|_h \leq K.
\end{equation}                                                      
Since $$\Delta_{g_h}u_h= D^+g_hD^+u_h+ g_hD^2u_h,$$
$$D^-\Delta_{g_h}u_h = D^2g_hD^-u_h+g_hD^3u_h+D^+g_hD^2u_h+D^-g_hD^-D^-u_h,$$
we deduce from (\ref{gronwall3}) and (\ref{gronwall4}) that sequences 
$\{\left|D^-\frac{du_h}{dt}\right|_h\}_h,$ $\{\left|\frac{du_h}{dt}\right|_h\}_h,$ $\{|D^2u_h|_h\}_h,$ and $\{|D^3u_h|_h\}_h$  are bounded
in $L^\infty(0,T_1).$ Thus the result follows from Lemma \ref{equiv_norms}.
\end{proof}
\subsubsection{Step 3}
We already proved, by Lemma (\ref{lem_born1}), that there exists $u\in L^\infty(0,T,H^1_{loc}(\mathbb{R}))$  and a subsequence
$\{u_h\}_h$ such that 
$$P_hD^-u_h \rightarrow \partial_xu\quad\text{in}\quad L^\infty(0,T,L^2(\mathbb{R}))\quad \text{weak star},$$
for all $T >0.$ In view of Lemma \ref{lem_born2}, there exist $v,w\in L^\infty(0,T_1,L^2(\mathbb{R}))$  and a subsequence $\{u_h\}_h$ such that 
\begin{equation}\left\{\begin{array}{lr}
P_hD^2u_h\rightarrow v\quad\text{in}\quad L^\infty(0,T_1,L^2(\mathbb{R}))\quad \text{weak star},&\\
P_hD^3u_h \rightarrow w\quad\text{in}\quad L^\infty(0,T_1,L^2(\mathbb{R}))\quad \text{weak star}.&\\
\end{array}\right.\end{equation}
Consequently, the sequence $\{\partial_xP_hD^-u_h\}_h$ converges to $\partial^2_xu$ in the sense of distributions.
On the other hand, $\partial_xP_hD^-u_h = Q_hD^2u_h,$ and the two sequences $\{Q_hD^2u_h\}_h$ and $\{P_hD^2u_h\}_h$ converge to the same limit
in $L^\infty(0,T_1,L^2(\mathbb{R}))$ weak star (Lemma \ref{lem_cv_l2f}). It follows that 
$\partial^2_xu = v \in L^\infty(0,T_1,L^2(\mathbb{R})),$ hence $\{P_hD^2u_h\}_h$ converges to
$\partial^2_xu$ in $L^\infty(0,T_1,L^2(\mathbb{R}))$ weak star. A similar argument shows that
$\partial^3_xu \in L^\infty(0,T_1,L^2(\mathbb{R}))$ and thus the proof is completed.

\subsection{Proof of Theorem \ref{th_un}}
First, we establish the two following lemmas 
\begin{lemma}\label{Lem_reg_1}
Let $g\in W^{1,\infty}(\R^+,\mathbb{R})$ be such that $g\geq\alpha$ for some $\alpha>0$.
Let $T>0$ and $u: [0,T]\times\mathbb{R}\rightarrow S^2$ be a solution to (\ref{LIA}) such that 
$\partial_x u\in L^\infty(0,T,H^1(\mathbb{R})).$ Then there exist  $C_1, C_2 >0$ depending on $g$ and
 $\|\partial_xu(0,.)\|_{H^1(\mathbb{R})}$ such that for almost every $t\in ]0,T[$ we have 
\begin{equation}\label{gronwall01}
\|\partial_tu\|^2_{L^2(\mathbb{R})}+\|\Delta_gu\|^2_{L^2(\mathbb{R})}\leq
C_1+C_2\int_0^t\left(\|\partial_t u(\tau)\|^2_{L^2(\mathbb{R})}+\|\Delta_gu(\tau)\|^2_{L^2(\mathbb{R})}\right)d\tau.
\end{equation}
\end{lemma}
\begin{proof}
Taking the $L^2-$scalar product in (\ref{LIA}) with $\Delta_gu$ and integrating by parts, we obtain
$$\frac{d}{dt}\int_{\mathbb{R}}g(x)|\partial_x u|^2dx=\int_{\mathbb{R}}\partial_tg(x)|\partial_x u|^2dx,$$
which gives 
\begin{equation}\label{inegG}
 \|\partial_xu(t,.)\|_{L^2(\mathbb{R})}\leq 
\sqrt{\frac{\|g\|_{L^\infty}}{\alpha}}\|\partial_xu(0,.)\|_{L^2(\mathbb{R})}\exp\left(\frac{\|\partial_tg\|_{L^\infty}}{2\alpha}\right), 
\forall t\in[0,T].
\end{equation}
Since $2u\cdot\partial_xu=\partial_x|u|^2=0,$ and by deriving (\ref{LIA}) with respect to $t$, we obtain
\begin{eqnarray}\label{deriv1}
\partial^2_tu&=& (u\wedge \Delta_g u)\wedge \Delta_g u+u\wedge \Delta_g(u\wedge \Delta_g u)+u\wedge \Delta_{\partial_tg} u\notag\\
&=& (u\cdot\Delta_g u)\Delta_g u - |\Delta_g u|^2u+u\wedge (\Delta_g u\wedge\Delta_g u+
2g\partial_xu\wedge\partial_x\Delta_gu+u\wedge\Delta_g^2u)+u\wedge \Delta_{\partial_tg}\notag\\
&=& (u\cdot\Delta_g u)\Delta_g u - |\Delta_g u|^2u+2g(u\cdot\partial_x\Delta_g u)\partial_xu
+ (u\cdot\Delta_g^2u)u -\Delta_g^2u+u\wedge \Delta_{\partial_tg}u.
\end{eqnarray}
It is clear that $\partial_tu\cdot u=0,$ then taking the $L^2-$scalar product in (\ref{deriv1}) with $\partial_tu$, we get  
\begin{eqnarray*}
 \frac{d}{dt}\int_{\mathbb{R}}(|\partial_tu|^2+|\Delta_gu|^2)dx &=&
4\int_{\mathbb{R}}g(u\cdot\partial_x\Delta_g u)(\partial_x u\cdot\partial_tu)dx\\
&& +2\int_{\mathbb{R}}(u\wedge \Delta_{\partial_tg}u)\cdot (u\wedge \Delta_gu)dx\\
&& +2\int_{\mathbb{R}}\Delta_{\partial_tg}u\cdot \Delta_gudx.
\end{eqnarray*}

Furthermore, we have
\begin{eqnarray}\label{deriv2}
u\cdot\partial_x\Delta_g u&=& \partial_x(u\cdot\Delta_g u)-\partial_x u\cdot\Delta_g u\notag\\
&=& -\frac{3}{2}\partial_x(g|\partial_xu|^2)-\frac{1}{2}\partial_xg|\partial_x u|^2,
\end{eqnarray}
and 
\begin{eqnarray}\label{deriv2A}
(u\wedge \Delta_{\partial_tg}u)\cdot (u\wedge \Delta_gu) &=& \Delta_{\partial_tg}u\cdot \Delta_gu-(u\cdot \Delta_{\partial_tg}u)(u\cdot \Delta_gu)\notag\\
&=& \Delta_{\partial_tg}u\cdot \Delta_gu-\frac{1}{2}\partial_tg^2|\partial_x u|^4.
\end{eqnarray}
Then, integrating by parts, we get 
\begin{eqnarray}\label{Iu}
\frac{1}{2}\frac{d}{dt}\int_{\mathbb{R}}(|\partial_tu|^2+|\Delta_g|^2)dx&=&
\frac{3}{4}\frac{d}{dt}\int_{\mathbb{R}}g^2|\partial_xu|^4dx-\int_{\mathbb{R}}g\partial_xg|\partial_xu|^2
(\partial_x u\cdot\partial_t u)dx\notag\\
&& -\frac{5}{4}\int_{\mathbb{R}}\partial_tg^2|\partial_xu|^4dx+2\int_{\mathbb{R}}\Delta_{\partial_tg}u\cdot \Delta_gudx
\end{eqnarray}
We denote $$I(u)=\|\partial_t u\|^2_{L^2(\mathbb{R})}+\|\Delta_g u\|^2_{L^2(\mathbb{R})}-\frac{3}{2}\int_{\mathbb{R}}g^2|\partial_x u|^4dx,$$
$$J(u)= -\int_{\mathbb{R}}g\partial_xg|\partial_xu|^2(\partial_x u\cdot\partial_t u)dx
 -\frac{5}{4}\int_{\mathbb{R}}\partial_tg^2|\partial_xu|^4dx+2\int_{\mathbb{R}}\Delta_{\partial_tg}u\cdot \Delta_gudx.$$
Relation (\ref{Iu}) can be rewritten as
\begin{equation}\label{Iu1}
\|\partial_t u\|^2_{L^2(\mathbb{R})}+\|\Delta_g u\|^2_{L^2(\mathbb{R})}=I(u(0,.))+\frac{3}{2}\int_{\mathbb{R}}g^2|\partial_x u|^4dx +2\int_0^tJ(u(\tau))d\tau.
\end{equation}
Then applying Gagliardo-Nirenberg inequalities on $\partial_x u,$ we get 
\begin{equation}\label{NBRG}\left\{\begin{array}{lr}
\|\partial_x u\|_{L^6(\mathbb{R})}\leq K_6\|\partial_xu\|^{\frac{2}{3}}_{L^2(\mathbb{R})}
\|\partial^2_x u\|^{\frac{1}{3}}_{L^2(\mathbb{R})},&\\
\|\partial_x u\|_{L^4(\mathbb{R})}\leq K_4\|\partial_xu\|^{\frac{3}{4}}_{L^2(\mathbb{R})}
\|\partial^2_xu\|^{\frac{1}{4}}_{L^2(\mathbb{R})},&\\
\end{array}
\right.
\end{equation}
with $K_6,K_4>0.$ On the other hand, we have 
\begin{equation}\label{Iu2}
 \|g\partial^2_x u\|^2_{L^2(\mathbb{R})}\leq 2\|\Delta_gu\|^2_{L^2(\mathbb{R})}+
2\|\partial_xg\partial_x u\|^2_{L^2(\mathbb{R})}.
\end{equation}
To find a suitable estimate for $I(u(0,.))$, we use the relation 
$$|\partial _tu|^2=|u\wedge\Delta_g u|^2=|\Delta_g u|^2-g^2|\partial_x u|^4,$$
which implies that  
\begin{eqnarray}\label{Iu3}
I(u(0,.))&=& 2\|\Delta_g u(0,.)\|^2_{L^2(\mathbb{R})}-\frac{5}{2}\int_{\mathbb{R}}g^2|\partial_xu(0,.)|^4dx\notag\\
&\leq& 2\|\Delta_g u(0,.)\|^2_{L^2(\mathbb{R})}+\frac{5}{2}K_4^4\|\partial_xu(0,.)\|^3_{L^2(\mathbb{R})}
\|\partial^2_xu(0,.)\|_{L^2(\mathbb{R})}.
\end{eqnarray}
Thus, inequalities (\ref{inegG}), (\ref{NBRG}), (\ref{Iu2}) and (\ref{Iu3}) together with $g\in W^{1,\infty}(\R^+, \R)$ allow, by using 
H\"{o}lder inequality, to upper-bound the right hand side of (\ref{Iu1}) by 
$$C_1+C_2\int_0^t\left(\|\partial_t u(\tau)\|^2_{L^2(\mathbb{R})}+\|\Delta_gu(\tau)\|^2_{L^2(\mathbb{R})}\right)d\tau,$$
where the two constants $C_1$ and $C_2$  depend on $g$ and $\|\partial_xu(0,.)\|_{H^1(\R)}.$

\end{proof}
\begin{corollary}\label{cor_reg_1}
Under the assumptions of lemma \ref{Lem_reg_1}, we have for almost every $t\in ]0, T[$ 
$$\|\partial^2_x u(t,.)\|^2_{L^2(\mathbb{R})}\leq D_1e^{D_2t},$$
where $D_1$ and $D_2$ are two  positive constants depending on $g$ 
and $\|\partial_xu(0,.)\|_{H^1(\mathbb{R})}.$
\end{corollary}
\begin{proof}
Let $$\psi(t)=\|\partial_t u(t)\|^2_{L^2(\mathbb{R})}+\|\partial^2_x u(t)\|^2_{L^2(\mathbb{R})}.$$
Inequality (\ref{gronwall01}) implies that
$$\psi(t) \leq C_1+C_2\int_0^t\psi(\tau)d\tau,$$
then conclusion follows from Gr\"{o}nwall lemma.
\end{proof}

\begin{lemma}\label{Lem_reg_2}
Let $g\in W^{1,\infty}(\R^+, W^{3,\infty}(\mathbb{R}))$ be such that $g\geq\alpha$ for some $\alpha>0$. 
Let $T>0$ and $u: [0,T]\times\mathbb{R}\rightarrow S^2$ be a solution to (\ref{LIA}) such that 
$\partial_x u\in L^\infty(0,T,H^2(\mathbb{R})).$ Then there exist $C_1, C_2 >0$ depending on $g$ and 
$\|\partial_xu(0,.)\|_{H^2(\mathbb{R})}$ such that for almost every $t\in ]0,T[$ we have 
$$\|\partial_t\partial_x u\|^2_{L^2(\mathbb{R})}+\|\partial^3_x u\|^2_{L^2(\mathbb{R})}\leq
C_1+C_2\int_0^t\left(\|\partial_t\partial_x u(\tau)\|^2_{L^2(\mathbb{R})}+\|\partial^3_x u(\tau)\|^2_{L^2(\mathbb{R})}\right)d\tau.$$
\end{lemma}
\begin{proof}
Since 
\begin{equation} \label{deriv3}
u\cdot\Delta_g^2u=\Delta_g(u\cdot\Delta_gu)-2g\partial_x u\cdot\partial_x\Delta_g u-|\Delta_g u|^2,
\end{equation}
we get by combining (\ref{deriv1}), (\ref{deriv2}) and (\ref{deriv3}) 
\begin{eqnarray}\label{deriv4}
\partial^2_tu+\Delta_g^2u&=& u\wedge\Delta_{\partial_tg} u-\Delta_g(g|\partial_x u|^2)-g|\partial_x u|^2\Delta_gu
-2\partial_x(g|\partial_x u|^2)\partial_x u \notag\\
&& -2g(\partial_x u\cdot\Delta_g u)\partial_x u-2g(\partial _xu\cdot\partial_x\Delta_g u)u
-2|\Delta_g u|^2\notag\\
&=& u\wedge\Delta_{\partial_tg} -\Delta_g(g|\partial_x u|^2u)-2\partial_x(g(\partial_x u\cdot\Delta_g u)u)\notag\\
&=& u\wedge\Delta_{\partial_tg} -2\Delta_g\left(|\partial_x u|^2u\right)
+\partial_x\left(|\partial_x u|^2(g\partial_x u-\partial_xgu)\right).
\end{eqnarray}
Deriving (\ref{deriv4}) with respect to $x$ and taking the $L^2-$scalar product with $g\partial_t\partial_x u,$ we get  
by integrating by parts
\begin{eqnarray}\label{Int1}
\frac{1}{2}\frac{d}{dt}\int_{\mathbb{R}}g\left(|\partial_t\partial_x u|^2+|\partial_x\Delta_gu|^2\right)dx &=&
-2\int_{\mathbb{R}}g\partial_x\Delta_g\left(|\partial_x u|^2u\right)\cdot\partial_t\partial_x udx\notag\\
&&+\int_{\mathbb{R}}g\partial^2_x\left(|\partial_x u|^2(g\partial_x u-g'u)\right)\cdot\partial_t\partial_x udx\notag\\
&&+\int_{\mathbb{R}}g\partial_x\left(u\wedge\Delta_{\partial_tg}u\right)\cdot\partial_t\partial_x udx\notag\\
&&+\int_{\mathbb{R}}g\partial_x\Delta_{\partial_tg}u\cdot\partial_x\partial_t\Delta_{g}udx 
+\int_{\mathbb{R}}\partial_tg|\partial_x\Delta_gu|^2dx.
\end{eqnarray}
We estimate the $L^2$ norm of the right-hand side  of (\ref{deriv4}) by applying the chain rule on the operators $\partial_x\Delta_g$
and $\partial^2_x$. All the terms of the right hand side  of (\ref{Int1}) except for  
$$J_1=-2\int_{\mathbb{R}}g^3\partial^3_x(|\partial_x u|^2)u\cdot\partial_t\partial_x udx,$$
can be upper-bounded  by $C\left(\|\partial_t\partial_x u(\tau)\|^2_{L^2(\mathbb{R})}+\|\partial^3_x u(\tau)\|^2_{L^2(\mathbb{R})}\right).$ 
 To estimate $J_1,$ we integrate by parts.  Hence we get 
$$J_1=2\int_{\mathbb{R}}\partial^2_x(|\partial_x u|^2)\partial_x(g^3u\cdot\partial_t\partial_x u)dx,$$
then we develop
$$\partial_x(u\cdot\partial_t\partial_x u)=\partial_x u\cdot\partial_t\partial_x u
+u\cdot\partial_t\partial^2_x u=\partial_x u\cdot\partial_t\partial_x u
+u\cdot\partial_t\partial^2_x u-\partial^2_x(u\cdot\partial_t u)
= -\partial_x u\cdot\partial_t\partial_x u-\partial^2_x u\cdot\partial_t u.$$
Thus we get 
$$J_1=6\int_{\mathbb{R}}g'g^2\partial^2_x(|\partial_x u|^2)u\cdot\partial_t\partial_x udx
-2\int_{\mathbb{R}}g^3\partial^2_x(|\partial_x u|^2)
(\partial_x u\cdot\partial_t\partial_x u+\partial^2_x u\cdot\partial_tu)dx,$$
and the conclusion holds from H\"{o}lder inequality and Sobolev embedding.
\end{proof}
\begin{corollary}\label{cor_reg_2}
Under the assumptions of Lemma \ref{Lem_reg_2}, we have for almost every  $t\in ]0,T[$ 
$$\|\partial^3_x u(t,.)\|^2_{L^2(\mathbb{R})}\leq D_1e^{D_2t},$$
where $D_1$ and $D_2$ are two positive constants depending on $g$ and $\|\partial_x u(0,.)\|_{H^2(\mathbb{R})}.$
\end{corollary}
\begin{proof}
The proof is an immediate result of Gr\"{o}nwall lemma.
\end{proof}

\subsubsection{Proof of Theorem \ref{th_un}}
We denote $\omega= u-\tilde{u}$ and $\omega_0= u_0-\tilde{u}_0.$
In what follows, we prove that there exist $C_k>0,\quad k=1,..,5,$ depending on $g$ and the $H^2$ norm of $\frac{d\tilde{u}_0}{dx}$ and $\frac{du_0}{dx},$ 
such that for almost every $t\in ]0,T_1[$ we have
\begin{equation}\label{ineg_un1}
\|\omega\|^2_{H^1(\mathbb{R})}\leq C_1\|\omega_0\|^2_{H^1(\mathbb{R})}+C_2\int_0^t\|\omega(\tau)\|^2_{H^1(\mathbb{R})}d\tau,
\end{equation}
and
\begin{eqnarray}\label{ineg_un2}
\|\partial_t \omega\|^2_{L^2(\mathbb{R})}+\|\partial^2_x\omega\|^2_{L^2(\mathbb{R})}&\leq&
C_3\| \omega_0\|^2_{H^1(\mathbb{R})}
+C_4\left(\|\partial_t \omega|_{t=0}\|^2_{L^2(\mathbb{R})}+\|\partial^2_x\omega_0\|^2_{L^2(\mathbb{R})}\right)\notag\\
&& + C_5\int_0^t\left(\|\partial_t \omega(\tau)\|^2_{L^2(\mathbb{R})}+\|\partial^2_x\omega(\tau)\|^2_{L^2(\mathbb{R})}\right)d\tau.
\end{eqnarray} 
Applying (\ref{LIAg}) and (\ref{deriv4}) on $u$ and $\tilde{u}$ and subtracting, we get
\begin{equation}\label{deriv5}
\partial_t\omega=z\wedge\Delta_g\omega+\omega\wedge\Delta_gz,
\end{equation}
and
\begin{eqnarray}\label{deriv6}
\partial^2_t\omega+\Delta_g^2\omega&=& z\wedge\Delta_{\partial_tg}\omega+\omega\wedge\Delta_{\partial_tg}z
-2\Delta_g(gQ\omega)+
\partial_x\left(Q(g\partial_x\omega-\partial_xg\omega)\right)\notag\\
&& - 4\Delta_g\left(g(\partial_x z\cdot\partial_x\omega)z\right)
+2\partial_x\left((\partial_x z\cdot\partial_x\omega)(g\partial_x z-\partial_xgz)\right),
\end{eqnarray}
with $z=\frac{1}{2}(u+\tilde{u})$ and $Q=\frac{1}{2}\left(|\partial_xu|^2+|\partial_x \tilde{u}|^2\right).$
Multiplying (\ref{deriv5}) by $\omega,$ we find that \\ $|\omega|^2=2(z\wedge\Delta_g\omega)\cdot\omega,$ which means that
$\omega\in L^2(\mathbb{R}).$ Then, integrating by parts and using H\"{o}lder's inequality, we get 
\begin{eqnarray}\label{eg_un1}
\frac{d}{dt}\int_{\mathbb{R}}|\omega|^2&=&-2\int_{\mathbb{R}}g(\omega\wedge\partial_x z)\cdot\partial_x \omega\notag\\
&\leq& 2\|g\partial_x z\|_{L^\infty(\mathbb{R})}\|\omega\|_{L^2(\mathbb{R})}\|\partial_x \omega\|_{L^2(\mathbb{R})}\notag\\
&\leq& \|g\partial_x z\|_{L^\infty(\mathbb{R})}\|\omega\|^2_{H^1(\mathbb{R})}.
\end{eqnarray}
Next, we take the $L^2-$scalar product in (\ref{deriv5}) with $\Delta_g\omega.$ Integrating by parts and using H\"{o}lder inequality,
we get 
\begin{eqnarray}\label{eg_un2}
\frac{d}{dt}\int_{\mathbb{R}}g|\partial_x\omega|^2&=&\int_{\R}\partial_t|\partial_x\omega|^2
-2\int_{\mathbb{R}}g\partial_x(\omega\wedge\Delta_g z)\cdot\partial_x \omega\notag\\
&=&\int_{\R}\partial_t|\partial_x\omega|^2-2\int_{\mathbb{R}}g(\omega\wedge\partial_x\Delta_g z)\cdot\partial_x \omega\notag\\
&\leq& (\|\partial_tg\|_{L^\infty(\R)}+\|g\partial_x \Delta_g z\|_{L^\infty(\mathbb{R})})\|\omega\|^2_{H^1(\mathbb{R})}.
\end{eqnarray}
Thus, (\ref{ineg_un1}) holds from Corollaries \ref{cor_reg_1} and \ref{cor_reg_2} and from Sobolev embedding
\footnote{There exists $C>0$ such that $$\|u\|_{L^\infty(\mathbb{R})}\leq C\|u\|_{H^1(\mathbb{R})}, \forall u\in H^1(\mathbb{R}).$$}
after summing (\ref{eg_un1}) and (\ref{eg_un2}).

Finally, taking the $L^2-$scalar product in (\ref{deriv6}) with $\partial_t\omega$ and integrating by parts, we get
$$\frac{1}{2}\frac{d}{dt}\int_{\mathbb{R}}(|\partial_t\omega|^2+|\Delta_g\omega|^2)=I_1+I_2+I_3-2E_1-4E_2+E_3+2E_4,$$
with
$$I_1= \int_{\R}\Delta_{\partial_tg}\omega\cdot\Delta_g\omega,$$
$$I_2= \int_{\R}z\wedge\Delta_{\partial_tg}\omega\cdot\partial_t\omega,\quad I_3= \int_{R}\omega\wedge\Delta_{\partial_tg}z\cdot\partial_t\omega,$$
$$E_1=\int_{\mathbb{R}}\Delta_g(gQ\omega)\cdot\partial_t\omega,\quad
E_3=\int_{\mathbb{R}}\partial_x\left(Q(g\partial_x\omega-g'\omega)\right)\cdot\partial_t\omega,$$
$$E_2=\int_{\mathbb{R}}\Delta_g\left(g(\partial_x z\cdot\partial_x\omega)z\right)
\cdot\partial_t\omega,\quad
E_4=\int_{\mathbb{R}}\partial_x\left((\partial_x z\cdot\partial_x\omega)(g\partial_x z-g'z)\right)\cdot\partial_t\omega.$$
The terms $I_1, I_2, I_3, E_1, E_3$ and $E_4$ can be estimated by applying H\"{o}lder's inequality and Sobolev's embedding 
$H^1(\mathbb{R})\subset L^\infty(\mathbb{R})$. Applying the chain rule on $\Delta_g$, the term $E_2$ can be written
$$E_2=\int_{\mathbb{R}}g^2(\partial_x z\cdot\partial^3_x\omega)(z\cdot\partial_t\omega)+E_{21},$$
where $E_{21}$ can be estimating by using H\"{o}lder inequality and Sobolev embedding. Finally, we have
$z\cdot\partial_t\omega=-\omega\cdot\partial_t z$ (since $|u|^2-|\tilde{u}|^2=0$) and  
$$
\int_{\mathbb{R}}g^2(\partial_x z\cdot\partial^3_x\omega)(z\cdot\partial_t\omega)
= -2\int_{\mathbb{R}}g'g(\partial_x z\cdot\partial^2_x\omega)
(z\cdot\partial_t\omega)+\int_{\mathbb{R}}g^2\partial^2_x\omega\cdot
\partial_x\left((\omega\cdot\partial_t z)\partial_x z\right),
$$
which is now in a suitable form to be estimated as above. This yields the desired claim at the $H^2$ level.

\subsection{Proof of Theorem \ref{ex_fcb_g(gamma)}}
We construct a solution $\gamma \in L^\infty(0,T_1,H^3_{loc}(\mathbb{R}))$ for the system
\begin{equation}\label{FCBg}
\left\{\begin{array}{lr}
\partial_t \gamma=g(t,x,\gamma)\partial_x\gamma\wedge \partial^2_x\gamma,&\\
\gamma(0,.)=\gamma_0.&
\end{array}\right.
\end{equation}
as a limit, when $h\rightarrow 0$ , of a sequence $\{\gamma_h\}_h$ of solutions to the semi-discrete system
\begin{equation}\label{FCBgD}
\left\{\begin{array}{lr}
\frac{d\gamma_h}{dt}=g_hD^+\gamma_h\wedge D^2\gamma_h,& t>0,\\
\gamma_h(0)=\gamma_h^0,&
\end{array}\right.
\end{equation}
where $\gamma_h^0=\{\gamma_h^0(x_i)\}_i\in (\mathbb{R}^3)^{\mathbb{Z}_h}$ is such that $|D^+\gamma_h^0(x_i)|=1,$ and $g_h=\{g(t,x_i,\gamma_h^0(x_i))\}_i.$
We denote $u_h= D^+\gamma_h, g_h^t= \partial_tg(t,x_i,\gamma(x_i))$ and $\Delta_{g_h}u_h = D^+(g_hD^-u_h).$ Then, applying $D^+$ on (\ref{FCBgD}), we get
\begin{equation}\label{LIA_FCB_d}
\frac{du_h}{dt}= u_h\wedge \Delta_{g_h}u_h.
\end{equation}
We have
\begin{eqnarray*}
\frac{d}{dt}\sum_i(g\gamma_h|D^-u_h|^2)(x_i)&=& \sum_i\frac{d\gamma_h(x_i)}{dt}\cdot\nabla_\gamma g(t,x_i,\gamma_h(x_i))|D^-u_h(x_i)|^2\\
&& + \sum_ig_h^t(x_i)|D^-u_h(x_i)|^2+\sum_i\left(g_hD^-u_h\cdot D^-\frac{du_h}{dt}\right)(x_i).
\end{eqnarray*}
Then, using Lemma \ref{IPPd}, we obtain
$$h\sum_i (g_hD^-u_h\cdot D^-\frac{du_h}{dt})(x_i) = -\left(\Delta_{g_h}u_h,\frac{du_h}{dt}\right)_h=0.$$
Thus, using $\frac{d\gamma_h}{dt}=g_hu_h\wedge D^-u_h,$ we can write
\begin{eqnarray}\label{estim8}
\frac{d}{dt}\sum_i(g_h|D^-u_h|^2)(x_i)&\leq& \|\nabla_{\gamma}g\|_{L^\infty}|D^-u_h|_{L^\infty_h}\sum_i(g_h|D^-u_h|^2)(x_i)\notag\\
&& + \|\partial_tg\|_{L^\infty}\sum_i|D^-u_h(x_i)|^2.
\end{eqnarray}
To get another estimate in $|\Delta_{g_h}|_h,$ we derive (\ref{LIA_FCB_d}) with respect to $t$. This yields
\begin{eqnarray}\label{deriv12}
\frac{d^2u_h}{dt^2}&=& \frac{du_h}{dt}\wedge \Delta_{g_h}uh+ u_h\wedge \frac{d}{dt}\Delta_{g_h}u_h\notag\\
&=& (u_h\wedge\Delta_{g_h}u_h)\wedge\Delta_{g_h}u_h \notag\\
&&+u_h\wedge \left(D^+\left(\frac{d\gamma_h}{dt}\cdot\nabla g(\gamma_h)D^-u_h\right)+\Delta_{g_h}\left(\frac{du_h}{dt}\right)+
\Delta_{g^t_h}u_h\right).
\end{eqnarray}
Next, we denote  
$$\tilde{\Delta}_{g_h}u_h=D^+\left(g_h(u_h\wedge D^-u_h\cdot\nabla g(\gamma_h))D^-u_h\right),$$
hence (\ref{deriv12}) becomes 
\begin{equation}
\frac{d^2u_h}{dt^2} = (u_h\wedge\Delta_{g_h}u_h)\wedge\Delta_{g_h}u_h + u_h\wedge \Delta_{g_h}(u_h\wedge\Delta_{g_h}u_h) 
+u_h\wedge (\tilde{\Delta}_{g_h}u_h+\Delta_{g^t_h}u_h).
\end{equation}
Repeating the same calculus as in (\ref{deriv7}), we get  
\begin{equation}\label{deriv13}
\frac{d^2u_h}{dt^2}+\Delta_{g_h}^2u_h=(u_h\cdot\Delta_{g_h}u_h )\Delta_{g_h}u_h-|\Delta_{g_h}u_h|^2u_h +(u_h\cdot\Delta_{g_h}^2u_h )u_h
+u_h\wedge (\tilde{\Delta}_{g_h}u_h+\Delta_{g^t_h}u_h)+ E,
\end{equation}
where 
\begin{eqnarray*}
E&=& \frac{h^2}{2}D^+[g_h(D^-u_h)^2D^-\Delta_{g_h}u_h]\\
&& -g_hD^-(A_{g_h}u_h)D^-u_h-g_h(D^- u_h\cdot\tau^-\Delta_{g_h}u_h)D^-u_h\\
&&-\tau^+g_hD^+(A_{g_h}u_h)D^+u_h-\tau^+g_h(D^+u_h\cdot\tau^+\Delta_{g_h}u_h)D^+u_h.
\end{eqnarray*}
Taking the $L_h^2-$scalar product in (\ref{deriv13}) with $\frac{du_h}{dt}$ and using both $u_h\cdot\frac{du_h}{dt}=0$ and
$\Delta_{g_h}u_h\cdot\frac{du_h}{dt}=0,$ we get by integration by parts 
$$\frac{1}{2}\frac{d}{dt}\left|\frac{du_h}{dt}\right|_h^2+\left(\Delta_{g_h}u_h,\Delta_{g_h}\left(\frac{du_h}{dt}\right)\right)_h=I+
\left(u_h\wedge \tilde{\Delta}_{g_h}u_h,\frac{du_h}{dt}\right)_h.$$
where $I = \left(E,\frac{du_h}{dt}\right)_h.$ We have  
$$\Delta_{g_h}\left(\frac{du_h}{dt}\right)= \frac{d}{dt}\Delta_{g_h}u_h-\tilde{\Delta}_{g_h}u_h-\Delta_{g^t_h}u_h.$$
Consequently,
$$\frac{1}{2}\frac{d}{dt}\left(\left|\frac{du_h}{dt}\right|_h^2+|\Delta_{g_h}u_h|_h\right)=I
+\left(\tilde{\Delta}_{g_h}u_h+\Delta_{g^t_h}u_h,\Delta_{g_h}u_h\right)_h 
+\left(u_h\wedge(\tilde{\Delta}_{g_h}u_h+\Delta_{g^t_h}u_h),u_h\wedge\Delta_{g_h}u_h\right)_h.$$
We know that $g_h$ and $D^+g_h$ are bounded in $L^\infty(0,T,L_h^\infty)$ by $\beta = \|g\|_{L^\infty(0,T,L^\infty)}$ and \\
 $\beta' = \|\partial_xg\|_{L^\infty(0,T,L^\infty)}+\|\nabla_\gamma g\|_{L^\infty(0,T,L^\infty)}$ respectively. Thus by following
the same calculus  in the proof of  Theorem \ref{th_ex_npf}, we find that there exists $C_1= C_1(\alpha ,\beta, \beta')>0$ such that
\begin{equation}\label{estim9}
I\leq C_1|D^+u_h|_{L^\infty_h}^2(|\Delta_{g_h}u_h|_h^2+|D^+u_h|_h^2+\left|\frac{du_h}{dt}\right|_h^2).
\end{equation}
To find a suitable estimate for the term $\left(\tilde{\Delta}_{g_h}u_h,\Delta_{g_h}u_h\right)_h,$ we first rewrite
\begin{eqnarray*}
\tilde{\Delta}_{g_h}u_h &=&D^+\left(g_h(u_h\wedge D^-u_h\cdot\nabla_\gamma g(\gamma_h))D^-u_h\right)\\
&=&(u_h\wedge D^-u_h.\nabla_\gamma g(\gamma_h))\Delta_{g_h}u_h +\tau^+g_hD^+(u_h\wedge D^-u_h.\nabla_\gamma g(\gamma_h))D^+u_h \\
&=&(u_h\wedge D^-u_h.\nabla_\gamma g(\gamma_h))\Delta_{g_h}u_h +\tau^+g_h(u_h\wedge D^2u_h.\nabla_\gamma g(\gamma_h))D^+u_h\\
&& + \tau^+g_h\left(u_h\wedge D^-u_h.D^+(\nabla_\gamma g(\gamma_h))\right)D^+u_h.
\end{eqnarray*}
The term $D^+(\nabla_\gamma g(\gamma_h))$ is bounded in  $L^\infty(0,T,L_h^\infty)$ by 
$\beta''= \|\partial_x\nabla_\gamma g\|_{L^\infty(0,T,L^\infty)}+\|\nabla^2_\gamma g\|_{L^\infty(0,T,L^\infty)}.$
It follows that  
\begin{eqnarray}\label{estim10A}
\left(\tilde{\Delta}_{g_h}u_h,\Delta_{g_h}u_h\right)_h &\leq& \beta ' |D^-u_h|_{L^\infty_h}|\Delta_{g_h}u_h|_h^2+
\beta ' |D^+u_h|_{L^\infty_h}|\tau^+D^2u_h|_h|\Delta_{g_h}u_h|_h\notag \\
&& +\beta \beta'' |D^+u_h|_{L^\infty_h}|D^+u_h|_h.
\end{eqnarray}
Furthermore, we have $\Delta_{g^t_h}u_h =D^+g^t_hD^-u_h+\tau^+g^t_hD^2u_h,$ then the two terms  
 $\tau^+g^t_h$ and $D^+g^t_h$ are bounded in  $L^\infty(0,T,L_h^\infty)$ by $\beta_1 = \|\partial_tg\|_{L^\infty(0,T,L^\infty)}$
and $\beta'_1 = \|\partial_t\partial_xg\|_{L^\infty(0,T,L^\infty)}+\|\partial_t\nabla_\gamma g\|_{L^\infty(0,T,L^\infty)}$ respectively.
Then we have 
\begin{equation}\label{estimgama}
 \left(\Delta_{g^t_h}u_h,\Delta_{g_h}u_h\right)_h \leq (\beta'_1|D^-u_h|_h+\beta_1|D^2u_h|_h)|\Delta_{g_h}u_h|_h.
\end{equation}
Using inequality $|\tau^+D^2u_h|_h  \leq |\Delta_{g_h}u_h|_h +\beta'|D^+u_h|_h$ together with 
 (\ref{estim8}), (\ref{estim9}), (\ref{estimgama}) and (\ref{estim10A}), we find 
that there exists $C= C(\alpha, \beta,\beta_1, \beta',\beta'_1, \beta'')$ such that
\begin{eqnarray*}
\frac{d}{dt}\left(\left|\frac{du_h}{dt}\right|_h^2+|\Delta_{g_h}u_h|_h+h\sum_i[g_h|D^+u_h|^2](x_i)\right)
&\leq& C\left(|D^+u_h|_{L^\infty_h}^2+|D^+u_h|_{L^\infty_h}\right)\\
&& \times \left(|\Delta_{g_h}u_h|_h^2+|D^+u_h|_h^2 +|D^+u_h|_h+\left|\frac{du_h}{dt}\right|_h^2\right).
\end{eqnarray*}
In view of Lemma \ref{Sobolev_discret}, there exists $\tilde{C}>0$ and $C= C(\alpha, \beta')>0$ such that
$$|D^+u_h|^2_{L^\infty_h}\leq C|D^+u_h|^2_{H^1_h}\leq C\tilde{C}(|\Delta_{g_h}u_h|_h^2+|D^+u_h|_h^2).$$
This implies the existence of two constants $C_1, C_2 >0$ depending on $\alpha,$ $\beta,$ $\beta_1,$ $\beta'$ $\beta'_1,$  and $\beta''$ such that
\begin{equation}\label{estim11}
\frac{d}{dt}\left(\left|\frac{du_h}{dt}\right|_h^2+|\Delta_{g_h}u_h|_h+h\sum_i[g_h|D^+u_h|^2](x_i)\right)\leq
C_1\left(|\Delta_{g_h}u_h|_h^2+|D^+u_h|_h^2+\left|\frac{du_h}{dt}\right|_h^2\right)^2+C_2.
\end{equation}
We construct now the sequence $\{\gamma_h^0\}$ such that 
\begin{equation}\label{cond_di1_gamma}\left\{\begin{array}{lr}
Q_h\gamma_h^0\rightarrow \gamma_0\quad\text{in}\quad L_{loc}^2(\mathbb{R}),&\\
Q_hD^+\gamma_h^0 \rightarrow \frac{d\gamma_0}{dx}\quad\text{in}\quad L_{loc}^2(\mathbb{R}),&\\
Q_hD^2\gamma_h^0 \rightarrow \frac{d^2\gamma_0}{dx^2}\quad\text{in}\quad L^2(\mathbb{R}),&\\
Q_hD^3\gamma_h^0 \rightarrow \frac{d^3\gamma_0}{dx^3}\quad\text{in}\quad L^2(\mathbb{R}).&\\
\end{array}\right.\end{equation}
Thus we have 
\begin{lemma}\label{lem_born3}
There exists $T_1>0$ such that
\begin{enumerate}
 \item The two sequences $\{P_h\gamma_h\}_h$ and $\{P_hu_h\}_h$ are bounded in $L^\infty(0,T_1,H^1_{loc}(\mathbb{R}))$.
 \item The sequences $\{\partial_tP_hu_h\}_h,$ $\{\partial_tP_h\gamma_h\}_h,$ $\{P_hD^+u_h\}_h$ and $\{P_hD^2u_h\}_h$ are bounded in $L^\infty(0,T_1,L^2(\mathbb{R}))$.
\end{enumerate}
\end{lemma}
\begin{proof}
Following the same steps in the proof of Lemma \ref{lem_born2}, we find that there exists $T_1>0$ and $M>0$ such that
for almost every $t\in ]0, T_1[,$ we have  
\begin{equation}\label{estim12}
|D^+u_h|_h^2+|D^2u_h|_h^2 +\left|\frac{du_h}{dt}\right|_h^2\leq M.
\end{equation}
To prove 1, let  $L>0.$  We denote $N = E(\frac{L}{h})+1.$ Since 
\begin{equation}\label{estim122}
\|P_h\gamma_h(t)\|_{H^1(-L,L)}\leq \sqrt{2L}|P_h\gamma_h(t,0)|+(2L+1)\|\partial_xP_h\gamma_h(t)\|_{L^2(-L,L)},
\end{equation}
and
\begin{eqnarray*}
|P_h\gamma_h(t,0)| &=& |\gamma_h(t,0)|\\
&\leq& |\gamma_h(0,0)|+ T_1 \left\|\frac{d}{dt}\gamma_h(.,0)\right\|_{L^\infty(0,T_1)} \\
&\leq& |\gamma_h(0,0)|+ T_1\beta \|D^-u_h(.,0)\|_{L^\infty(0,T_1)} \\
&\leq& |\gamma_h(0,0)|+ T_1\beta \sup_{\tau \in [0,T]}|D^-u_h(\tau,.)|_{L^\infty_h}, \\
\end{eqnarray*}
inequality (\ref{estim12}) together with Lemma \ref{Sobolev_discret} imply the existence of a constant $C>0$ such that
\begin{equation}\label{estim13}
|P_h\gamma_h(t,0)| \leq |\gamma_h^0(0)|+ CT_1\beta\sqrt{M},
\end{equation}
for almost every $t\in ]0, T_1[.$ To estimate the second term of the right-hand side of (\ref{estim122}), we write
\begin{equation}\label{estim14}
\|\partial_xP_h\gamma_h\|^2_{L^2(-L,L)} = \sum_{i=-N}^{N-1}\int_{x_i}^{x_{i+1}}|D^+\gamma_h(x_i)|^2dx \leq 2L+h,
\end{equation}
hence we find that for almost every $t\in ]0, T_1[,$  
\footnote{It is possible to define $\{\gamma_h^0\}_h$ by $$\gamma_h^0(x_i)=\gamma_0(x_i),\forall i\in \mathbb{Z},$$ hence $\gamma_h^0(0)=\gamma_0(0).$ }
$$\|P_h\gamma_h(t)\|_{H^1(-L,L)}\leq \sqrt{2L}(|\gamma_0(0)|+ CT_1\beta\sqrt{M})+(2L+1)^2.$$
On the other hand, we have
\begin{eqnarray*}
\|P_hu_h\|^2_{H^1(-L,L)}&=&\sum_{i=-N}^{N-1}\int_{x_i}^{x_{i+1}}\left|\frac{x_i-x}{h}u_h(x_i)+\frac{x-x_i}{h}u_h(x_{i+1})\right|^2dx+
\sum_ih\left|\frac{u_h(x_i)-u_h(x_{i+1})}{h}\right|^2dx\\
&\leq&\sum_{i=-N}^{N-1}\frac{h}{3}\left(|u_h(x_i)|^2+|u_h(x_{i+1})|^2+u_h(x_i)u_h(x_{i+1})\right)+|D^+u_h|^2_h\\
&\leq& 2L+1+M.
\end{eqnarray*}
This completes the proof of 1.

Property 2 is an immediate result of (\ref{estim13}) and Lemma \ref{equiv_norms}.
\end{proof}

Lemma \ref{lem_born3} together with Lemma \ref{compa} imply the existence of $u, \gamma\in L^\infty(0,T_1,L^2_{loc}(\mathbb{R})),$ 
$\omega, v \in L^\infty(0,T_1,L^2(\mathbb{R})),$ and two subsequences $\{\gamma_h\}_h$ and $\{u_h\}_h$ such that
\begin{equation}\left\{\begin{array}{lr}
P_h\gamma_h \rightarrow \gamma \quad\text{in}\quad L^2(0,T_1,L^2_{loc}(\mathbb{R}))\quad \text{and almost everywhere},&\\
\partial_tP_h\gamma_h \rightarrow \partial_t\gamma \quad\text{in}\quad L^\infty(0,T_1,L^2(\mathbb{R}))
\quad \text{weak star},&\\
P_hu_h \rightarrow u \quad\text{in}\quad L^2(0,T_1,L^2_{loc}(\mathbb{R}))\quad \text{and almost everywhere},&\\
P_hD^-u_h \rightarrow v \quad\text{in}\quad L^\infty(0,T_1,L^2(\mathbb{R}))\quad \text{weak star},&\\
P_hD^2u_h \rightarrow w \quad\text{in}\quad L^\infty(0,T_1,L^2(\mathbb{R}))\quad \text{weak star}.&\\
\end{array}\right.\end{equation}
It follows that $\{\partial_xP_hu_h\}_h$ converges to $\partial_xu$ in the sense of distributions and, since
$\partial_xP_hu_h = Q_hD^+u_h,$ we also have $\partial_xu=v \in L^\infty(0,T_1,L^2(\mathbb{R})).$
We now prove that $\{P_h(g_hu_h\wedge D^-u_h)\}_h$ converges to $g(\gamma)u\wedge\partial_x u$ 
in $ L^\infty(0,T_1,L^2(\mathbb{R}))$ weak star. We first note that  
$$Q_h(g_hu_h\wedge D^-u_h)= g(Q_h\gamma_h)Q_hu_h\wedge Q_hD^-u_h.$$
This implies that the sequence $\{Q_h(g_hu_h\wedge D^-u_h)\}_h$ converges to $g(\gamma)u\wedge\partial_x u$ in $ L^\infty(0,T_1,L^2(\mathbb{R}))$
weak star. In view of Lemma \ref{lem_cv_l2f}, the two sequences $\{Q_h(g_hu_h\wedge D^-u_h)\}_h$ and $\{P_h(g_hu_h\wedge D^-u_h)\}_h$ converge
to the same limit. Since $\{\partial_tP_h\gamma_h\}_h$ converges to $\partial_t\gamma$ in
$ L^\infty(0,T_1,L^2(\mathbb{R}))$ weak star, we finally get
\begin{equation}
\partial_t\gamma=g(\gamma)u\wedge\partial_xu.
\end{equation}
Thus to complete this proof, it suffices to show that $\partial_x\gamma=u$ and that $\partial^2_x u\in L^\infty(0,T_1,L^2(\mathbb{R})).$
The sequence $\{\partial_xP_h\gamma_h\}_h$ converges to $\partial_x\gamma$ in the sense of distributions. On the other hand, we have 
$\partial_xP_h\gamma_h = Q_hD^+\gamma_h= Q_hu_h,$ and the sequence $\{Q_hu_h\}_h$ converges to $u$ in $L^\infty(0,T_1,L^2_{loc}(\mathbb{R}))$.
Indeed, for $L>0$ and $N= E(\frac{L}{h})+1,$ we have 
\begin{eqnarray*}
\|Q_hu_h-P_hu_h\|^2_{L^2(-L,L)}&\leq&\sum_{i=-N}^{N-1}\int_{x_i}^{x_{i+1}}|D^+u_h(x_i)|^2(x-x_i)^2dx\\
&\leq& \frac{2}{3}N|D^+u_h|^2_{L_h^\infty}h^3\\
&\le& \frac{2}{3}C(L+h)|D^+u_h|^2_{H^1_h}h^2\\
&\le& \frac{2}{3}CM(L+h)h^2,
\end{eqnarray*}
hence $$\partial_x\gamma=u.$$
The sequence $\{\partial_xP_hD^-u_h\}_h$ converges to $\partial^2_xu$ in the sense of distributions.
We have $\partial_xP_hD^-u_h = Q_hD^2u_h,$ and in view of Lemma \ref{lem_cv_l2f}, the two sequences $\{Q_hD^2u_h\}_h$ and $\{P_hD^2u_h\}_h$ 
converge to the same limit in $L^\infty(0,T_1,L^2(\mathbb{R}))$ weak star. Thus
$\partial^2_xu = w \in L^\infty(0,T_1,L^2(\mathbb{R})).$ 

\bibliographystyle{plain}
\bibliography{art_FPCB_biblio}
\end{document}